\documentclass[a4paper]{amsart}
\usepackage{txfonts, amsmath,amstext,amsthm,amscd,amsopn,verbatim,amssymb, amsfonts}
\usepackage{fullpage}

\usepackage[all]{xy}

\usepackage[bbgreekl]{mathbbol}
\usepackage{tikz}
\usetikzlibrary{matrix}
\usetikzlibrary{positioning, calc}
\usetikzlibrary{shapes}
\usetikzlibrary{arrows}  
\usetikzlibrary{calc,3d}
\usetikzlibrary{decorations,decorations.pathmorphing}
\usetikzlibrary{through}
\tikzset{ext/.style={circle, draw,inner sep=1pt},int/.style={circle,draw,fill,inner sep=1pt},nil/.style={inner sep=1pt}}
\tikzset{exte/.style={circle, draw,inner sep=3pt},inte/.style={circle,draw,fill,inner sep=3pt}}
\tikzset{diagram/.style={matrix of math nodes, row sep=3em, column sep=2.5em, text height=1.5ex, text depth=0.25ex}}
\tikzset{diagram2/.style={matrix of math nodes, row sep=0.5em, column sep=0.5em, text height=1.5ex, text depth=0.25ex}}
\theoremstyle{plain}
  \newtheorem{thm}{Theorem}
  \newtheorem{defi}{Definition}
  
  \newtheorem{prop}{Proposition}
  
  \newtheorem{cor}[prop]{Corollary}
  \newtheorem{lemma}{Lemma}
\theoremstyle{definition}
  \newtheorem{ex}{Example}
  \newtheorem{rem}{Remark}

\newcommand{\Hom}{\mathrm{Hom}}

\newcommand{\R}{{\mathbb{R}}}
\newcommand{\Q}{{\mathbb{Q}}}
\newcommand{\Z}{{\mathbb{Z}}}
\newcommand{\K}{{\mathbb{K}}}

\newcommand{\HGC}{{\mathrm{HGC}}}
\newcommand{\fHGC}{{\mathrm{fHGC}}}

\newcommand{\Graphs}{{\mathsf{Graphs}}}

\newcommand{\Gr}{{\mathsf{Gr}}}

\newcommand{\Def}{\mathrm{Def}}

\newcommand{\Poiss}{\mathsf{Pois}}
\newcommand{\Pois}{\Poiss}
\newcommand{\hoPois}{\mathsf{hoPois}}

\newcommand{\op}{\mathcal}

\newcommand{\Lie}{\mathsf{Lie}}

\newcommand{\hoLie}{\mathsf{hoLie}}

\newcommand{\coker}{\mathrm{coker}}

\newcommand{\Ass}{\mathsf{Assoc}}
\newcommand{\Com}{\mathsf{Com}}

\newcommand{\bpm}{\begin{pmatrix}}
\newcommand{\epm}{\end{pmatrix}}

\newcommand{\GC}{\mathrm{GC}}

\newcommand{\hoAss}{\mathsf{hoAss}}
\newcommand{\hoCom}{\mathsf{hoCom}}

\newcommand{\GL}{\mathrm{GL}}

\newcommand{\Out}{\mathrm{Out}}
\newcommand{\Aut}{\mathrm{Aut}}
\newcommand{\Fin}{\mathrm{Fin}}

\newcommand{\OS}{\mathrm{OS}} 

\newcommand{\OX}{\mathsf{O}}
\newcommand{\Grx}{\mathsf{Grx}}
\newcommand{\HP}{\mathsf{HH}}
\newcommand{\HPC}{\mathsf{CH}}

\newcommand{\ob}{\mathrm{ob}}
\newcommand{\Vect}{\mathsf{Vect}}
\newcommand{\G}{\mathsf{G}}
\newcommand{\Or}{\mathsf{Or}}
\newcommand{\colim}{\mathsf{colim}}

\newcommand{\linegraph}{
\begin{tikzpicture}[baseline=-.65ex]
\draw (0,0) -- (1,0);
\end{tikzpicture}
}
\newcommand{\tripodgraph}{
\begin{tikzpicture}[scale=.6,baseline=-.65ex]
\node[int] (v) at (0,0){};
\draw (v) -- +(90:1) (v) -- ++(210:1) (v) -- ++(-30:1);
\end{tikzpicture}
}

\newcommand{\Det}{\mathrm{Det}}

\newcommand{\beq}[1]{\begin{equation}\label{#1}}
\newcommand{\eeq}{\end{equation}}

\begin{document}
\title{Commutative hairy graphs and representations of $\Out(F_r)$}

\author{Victor Turchin}
\address{Department of Mathematics\\
  Kansas State University\\
  138 Cardwell Hall\\
  Manhatan, KS 66506, USA}
  \email{turchin@ksu.edu}
\author{Thomas Willwacher}
\address{Institute of Mathematics \\ University of Zurich \\  
Winterthurerstrasse 190 \\
8057 Zurich, Switzerland}
\email{thomas.willwacher@math.uzh.ch}

\thanks{V.T. acknowledges partial support by the MPIM, Bonn,  and the IHES. T.W. acknowledges partial support by the Swiss National Science Foundation (grant 200021\_150012) and the SwissMap NCCR, funded by the Swiss National Science Foundation}


\begin{abstract}
We express the hairy graph complexes computing the rational homotopy groups of long embeddings (modulo immersion) of $\R^m$ in $\R^n$ as ``decorated'' graph complexes associated to certain representations of the outer automorphism groups of free groups.
This interpretation gives rise to a natural spectral sequence, which allows us to shed some light on the structure of the hairy graph cohomology. We also explain briefly 
the connection to the deformation theory of the little discs operads and some conclusions that this brings.

\end{abstract}

\maketitle

\section{Introduction}
In this paper we consider complexes of linear combinations of isomorphism classes of graphs with external legs (or "hairs"), such as the following.
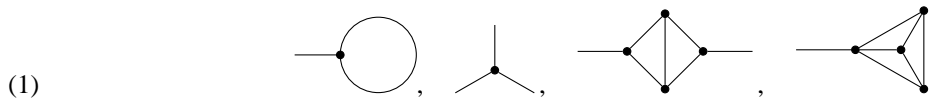
\begin{equation}\label{equ:HGCsample}
 \begin{tikzpicture}[scale=.5]
 \draw (0,0) circle (1);
 \draw (-180:1) node[int]{} -- +(-1.2,0);
 \end{tikzpicture}
,\quad
\begin{tikzpicture}[scale=.6]
\node[int] (v) at (0,0){};
\draw (v) -- +(90:1) (v) -- ++(210:1) (v) -- ++(-30:1);
\end{tikzpicture}
\,,\quad
\begin{tikzpicture}[scale=.5]
\node[int] (v1) at (-1,0){};\node[int] (v2) at (0,1){};\node[int] (v3) at (1,0){};\node[int] (v4) at (0,-1){};
\draw (v1)  edge (v2) edge (v4) -- +(-1.3,0) (v2) edge (v4) (v3) edge (v2) edge (v4) -- +(1.3,0);
\end{tikzpicture}
 \, ,\quad
 \begin{tikzpicture}[scale=.6]
\node[int] (v1) at (0,0){};\node[int] (v2) at (180:1){};\node[int] (v3) at (60:1){};\node[int] (v4) at (-60:1){};
\draw (v1) edge (v2) edge (v3) edge (v4) (v2)edge (v3) edge (v4)  -- +(180:1.3) (v3)edge (v4);
\end{tikzpicture}
 \, 
 \end{equation}
 The differential on these complexes is defined by summing over all ways of expanding a (non-hair-)vertex.
  More precisely, due to choices in signs and degrees, the hairy graph complexes come in several variants (which we denote $\HGC_{m,n}$), depending on a pair of integers $m,n$. For more details, see section \ref{sec:hairydef_std} below.  
   The hairy graph homology $H(\HGC_{m,n})$ is an object of significant interest in algebraic topology, since it computes the rational homotopy groups of spaces of long embeddings $\R^m\to \R^n$ modulo immersions in codimensions $n-m\geq 3$ as has been shown in \cite{FTW}.\footnote{For a weaker range of dimensions
this result has been shown earlier in~\cite{Turchin1,Turchin3}.}
However, our current knowledge of the hairy graph homology is rather limited.
Let us just note that since the differential cannot alter the number of loops or hairs of graphs, the complexes $\HGC_{m,n}$ split into finite dimensional subcomplexes $\HGC_{m,n}^{r, h}$ of fixed loop number $r$, and fixed number of hairs $h$.
Furthermore, up to unimportant degree shifts, the complexes $\HGC_{m,n}$ depend on $m$ and $n$ only through their parity, so that there are only four essentially different cases to consider.
Finally, the complexes $\HGC_{m,n}$ carry a natural Lie bracket.

The purpose of this paper is twofold. First, we shed some light on the structure of the hairy graph homology. Secondly, we show that the complexes $\HGC_{m,n}$ may be replaced by somewhat simpler quasi-isomorphic complexes.

To this end let us recall a more geometric approach to defining many types of graph complexes.
Culler and Vogtmann~\cite{CV} defined the so called {\it outer space} $\OS_r$ whose points are isomorphism classes of metrized $r$-loop graphs, i.e., graphs with a non-negative length assigned to each edge, such that the combined length of any closed loop is positive.\footnote{We note that our notation $\OS_r$ is slightly inconsistent with that Culler and Vogtmann, in that their outer space is contractible and comes with an action of $\Out(F_r)$, while our $\OS_r$ is the quotient under this action. Unfortunately, Culler and Vogtmann did not coin a catchy name for this quotient.}  A metrized graph is identified with the metrized graph obtained by contracting all edges of zero lengths.

The space $\OS_r$ plays a role similar to the classifying space of the group of outer automorphisms of the free group on $r$ generators $F_r$. In particular, a representation $V$ of $\Out(F_r)$ determines a local system on $\OS_r$.
For any such local system one may write down a "decorated" graph complex $\GC^r_V$ computing the compactly supported cohomology on $\OS_r$ with values in the local system.

Now we are ready to state our main results. We begin with the case of even codimension $n-m$. Let $\K$ be a field of characteristic zero. 
Abusing notation we also denote by $\K$ the trivial one-dimensional representation of $\Out(F_r)$.
We denote by $H_1$ the $r$-representation of $\Out(F_r)$ obtained by pulling back the canonical representation of $\GL(r,\Z)$ on $\K^r$ under the map $\Out(F_r)\to \GL(n,\Z)$. 
Finally, we define the one-dimensonal representation $\Det=\wedge^r H_1$.

\begin{thm}\label{thm:main_even}
Let  $n-m$ be even. The complex $\HGC_{m,n}^{r,h}$, $r\geq 2$, $h\geq 1$, admits a splitting into 
a direct sum of two complexes $\HGC_{m,n}^{r,h,I}\oplus\HGC_{m,n}^{r,h,II}$ such that in the homology
one gets the following splitting:
\begin{itemize}
\item if $n$ is even, then 
\beq{equ:thm1_1}
H(\HGC_{m,n}^{r,h}) \cong 
H(\GC^r_{S^h H_1})[nr+(h-1)(n-m-2)-2] \oplus H(\GC^r_{S^{h-1}H_1})[nr+(h-1)(n-m-2)-1] 
\eeq
where $S^hH_1$ is the $h$-fold symmetric power of the representation $H_1\cong \K^r$;
\item if $n$ is odd, then 
\beq{equ:thm1_2}
H(\HGC_{m,n}^{r,h}) \cong 
H(\GC^r_{\Det\otimes S^hH_1})[nr+(h-1)(n-m-2)-2] \oplus H(\GC^r_{\Det\otimes S^{h-1}H_1})[nr+(h-1)(n-m-2)-1] 
.
\eeq
\end{itemize}
 In both cases the Lie bracket with the graph 
\beq{equ:linegraph}
 L=\begin{tikzpicture}[baseline=-.65ex] \draw (0,0) -- (1,0); \end{tikzpicture}
\eeq
sends the subcomplex $\HGC_{m,n}^{r,h,I}$ isomorphically to $\HGC_{m,n}^{r,h+1,II}$  (in the homology it maps  the repeated  summands $H(\GC^r_{S^hH_1})$, or respectively, $H(\GC^r_{\Det \otimes S^hH_1})$, $h\geq 1$, identically
one  onto another)\footnote{Note that the Lie bracket with $L$ raises the number of hairs $h$ by one.} and sends $\HGC_{m,n}^{r,h,II}$ to zero.
\end{thm}

If the number $n-m$ is odd, the story is more complicated in that we need to consider representations of $\Out(F_r)$ that do not factor through $\GL(r,\Z)$. The representations we need are described in \cite{TW2}, from which we recall the following. 
Let $A$ be a commutative algebra.
Then the group $\Out(F_r)$ acts on the Hochschild-Pirashvili homology of $A$ on a wedge of $r$ circles $W_r$.
If $A$ is graded, then the $\Out(F_r)$ action restricts to each graded component of the Hochschild-Pirashvili homology.
We need the case of $A$ being the 2-dimensional graded algebra (the dual numbers)
\[
A= \K[x]/x^2 = \K 1 \oplus \K x 
\] 
with $x$ in degree $0$.
This algebra carries an auxiliary grading by assigning $x$ degree $1$.
We let $B_r^h$ be the respresentation of $\Out(F_r)$ on the piece of auxiliary degree $h$ and homological degree $h$ of the Hochschild-Pirashvili homology of $A$ on a wedge on $r$ circles.
These representations may be explicitly computed and have small dimensions, see \cite{TW2} for a more detailed discussion.\footnote{In fact, these examples yield the smallest known representations of $\Out(F_r)$ not factoring through $\GL(r,\Z).$} 
With this preparation we can state our second main result.

\begin{thm}\label{thm:main_odd}
Let $n-m$ be odd. The complex $\HGC_{m,n}^{r,h}$, $r\geq 2$, $h\geq 1$, admits a decreasing filtration (by defect)
\beq{eq:main_filt}
\HGC_{m,n}^{r,h}=F_0^{r,h}\supset F_1^{r,h}\supset F_2^{r,h}\supset\ldots\supset F_h^{r,h}
\supset F_{h+1}^{r,h}=0,
\eeq
such that all the terms $F_i^{r,h}/F_{i+1}^{r,h}$, $i\geq 2$, are acyclic.  Thus the first term $E_1$ of the spectral sequence associated with this filtration has only two columns $E_1^{0*}=H(F_0^{r,h}/F_{1}^{r,h})$ and $E_1^{1*}=H(F_1^{r,h}/F_2^{r,h})$ described as follows:
\begin{itemize}
\item if $n$ is even, then 
\beq{equ:thm2_1}
E_1^{0*}\oplus E_1^{1*} = H(\GC^r_{B^h_r})[nr+(h-1)(n-m-2)-2]\oplus 
\begin{cases}
H(\GC^r_{\K})[nr-1] & \text{for $h=1$} \\
0 & \text{for $h=2$} \\
 H(\GC^r_{B^{h-2}_r})[nr+(h-1)(n-m-2)-1] & \text{for $h\geq 3$}
\end{cases};
\eeq
\item if $n$ is odd, then 
\beq{equ:thm2_2}
E_1^{0*}\oplus E_1^{1*} = H(\GC^r_{\Det\otimes B^h_r})[nr+(h-1)(n-m-2)-2]\oplus 
\begin{cases}
H(\GC^r_{\Det})[nr-1] & \text{for $h=1$} \\
0 & \text{for $h=2$} \\
 H(\GC^r_{\Det\otimes B^{h-2}_r})[nr+(h-1)(n-m-2)-1] & \text{for $h\geq 3$}
\end{cases}.
\eeq
\end{itemize}
 The differential $d_1\colon E_1^{0*}\to E_1^{1*}$ is trivial for $h\leq 2$.  In both cases the Lie bracket with the tripod graph 
\beq{equ:tripodgraph}
 T=\begin{tikzpicture}[scale=.5,baseline=-.65ex] \draw (0,0) -- (90:1) (0,0)--(-30:1) (0,0)--(210:1); \end{tikzpicture}
\eeq
maps $F_i^{r,h}$ in $F_{i+1}^{r,h+2}$ thus inducing a map between  the corresponding spectral sequences. 
The induced map sends the column $E_1^{0*}$ of the $(r,h)$ component identically to the column $E_1^{1*}$
of the $(r,h+2)$ component (identifying the two repeated summands $H(\GC^r_{B^h_r})$, or respectively, 
$H(\GC^r_{\Det \otimes B^h_r})$, $h\geq 1$), and it sends the column $E_1^{1*}$ to zero.
%
\end{thm}

In fact, we conjecture that the spectral sequence abuts at the $E_1$ page, so that the hairy graph homology is in fact isomorphic to the expressions in the Theorem. In case of even $n-m$, the hairy complexes have a similar filtration by defect. The splitting of Theorem~\ref{thm:main_even} implies the collapse  of the associated spectral sequence at the first page $E_1$ in this case.

\subsection*{Structure of the paper}
In section \ref{sec:HHonrose} we recall some facts about the Hochschild-Pirashvili homology that will be crucial for our work. In section \ref{sec:outer} we define the decorated graph complexes. 
Finally, section \ref{sec:proofs} contains the proofs of our main results Theorems \ref{thm:main_even} and \ref{thm:main_odd}, along with some concluding remarks.

\subsection*{Acknowledgements}
We thank B. Fresse for helpful discussions. V.T. thanks the MPIM, Bonn, and the IHES, where he spent his sabbatical and where he was working on this project, for a partial support and hospitality. T.W. has been partially supported by the Swiss National Science foundation, grant 200021\_150012, and the SwissMAP NCCR funded by the Swiss National Science foundation.

\section{Notation}
We work over a ground field $\K$ of characteristic zero unless otherwise stated. All vector spaces are assumed to be vector spaces over the ground field $\K$. Graded vector spaces are vector spaces with a $\mathbb{Z}$-grading, and we abbreviate the phrase ``differential graded'' by dg as usual.
We generally use in cohomological conventions, i.e., the differentials will have degree $+1$. In particular the grading that we use for 
the hairy graph-complex is reversed compared to the grading of the rational homotopy groups of the spaces of long embeddings $\R^m\to\R^n$ modulo immersions.\footnote{In fact in the hairy graph-complex besides the grading reversion, we  also  shift degree by $m$ so that $\HGC_{m,n}$ is endowed with a natural Lie bracket corresponding to the Browder operator in the rational homotopy of the spaces of 
embeddings~\cite{FTW}.}

\section{Recollection: Higher Hochschild Homology}\label{sec:HHonrose}
Let $\Fin$ be the category of finite sets.
A right $\Fin$-module is a contravariant functor $\Fin\to dg\Vect$ into the category of dg vector spaces, and a left  $\Fin$-module is a covariant functor $\Fin\to dg\Vect$.
 
We will consider the following examples:
\begin{itemize}
\item For $X$ some topological space we can consider the right $\Fin$-module sending a finite set $S$ to the simplicial chains on the mapping space $C(X^S)$. We denote this $\Fin$-module by $C(X^\bullet)$.
\item To a commutative coalgebra $B$ we assign the right $\Fin$-module sending the finite set $S$ to the tensor product $B^S\cong \bigotimes_{s\in S} B$. We denote this $\Fin$-module by $B^\bullet$. Dually, if $A$ is a commutative algebra, then the functor $S\mapsto A^S\cong \bigotimes_{s\in S} A$ describes a left $\Fin$-module $A_\bullet$.
\end{itemize}

The higher Hochschild(-Pirashvili) homology $\HP^{X}(B)$ can be defined as the homology of the complex of homotopy natural transformations $C(X^\bullet) \to B^\bullet$~\cite{pirash00}. Dually, the higher Hochschild homology $\HP^{X}(A)$  may be described as derived tensor product $C(X^\bullet) \bar \otimes_{\Fin} A_\bullet$. We will provide explicit models below.
Any map $f:X \to Y$ induces a map $f^*: \HP^{Y}(B)\to \HP^X(B)$ (resp. $f_* : \HP^{X}(A)\to \HP^{Y}(A)$).
Two homotopic maps induce the same map in higher Hochschild homology.

In our case we take $X$ to be a rose with $r$ petals $W_r=\vee_r S^1$.
The outer automorphism group $\Out(F_r)$ acts on $W_r$ up to homotopy and hence we obtain a representation of $\Out(F_r)$ on $\HP^{W_r}(A)$ and $\HP^{W_r}(B)$ for all $A,B$ as above.

This yields a rich source of $\Out(F_r)$ representations, considered in \cite{TW2}.

\subsection{More explicit formulas for graphs}\label{sec:moreexplicit}
We assume that $A$ is an augmented differential graded commutative algebra, and denote by $\bar A$ the augmentation ideal. One can then describe $\HP^{W_n}(A)$ more explicitly as the homology of the complex 
\[
 \HPC^{W_r}(A) = \bigoplus_{j_1,\dots,j_r} A \otimes \bigotimes_{\alpha=1}^r (\overline A[1])^{\otimes j_\alpha},
\]
with differential $d_H+d_A$, where $d_A$ is induced from the intrinsic differential on $A$ and $d_H$ is a version of the Hochschild differential~\cite{TW2}. Concretely, for $a_0, a_{j,i}\in A$, for simplicity assumed to have degree 0, we have
\begin{align*}
 &d_H(a_0;a_{1,1},\dots,a_{1,j_1}, \dots ,a_{r,1},\dots,a_{r,j_r})
 \\
 &= 
 (a_0a_{1,1};a_{1,2},\dots,a_{1,j_1}, \dots ,a_{r,1},\dots,a_{r,j_r})
 -(a_0;a_{1,1}a_{1,2},\dots,a_{1,j_1}, \dots ,a_{r,1},\dots,a_{r,j_r})
 \\& \quad\quad\quad\quad\quad \pm \cdots + (-1)^{}(a_{1,j_1}a_0;a_{1,1},\dots,a_{1,j_1-1}, \dots ,a_{r,1},\dots,a_{r,j_r})
 \\
 &\quad \pm \cdots 
 \pm 
  (a_0a_{r,1}; a_{1,1},a_{1,2},\dots,a_{1,j_1}, \dots ,a_{r,2},\dots,a_{r,j_r})
 \pm (a_0;a_{1,1},a_{1,2},\dots,a_{1,j_1}, \dots ,a_{r,1}a_{r,2},\dots,a_{r,j_r})
 \\& \quad\quad\quad\quad\quad \pm \cdots \pm(a_{r,j_r}a_0;a_{1,1},\dots,a_{1,j_1-1}, \dots ,a_{r,1},\dots,a_{r,j_r-1}).
\end{align*}
Informally speaking, we may think of the $a$'s as sitting on a wedge of $r$ circles, and the differential is the signed sum of all contractions of spaces between the $a$'s, multiplying the two elements on either side of the space.
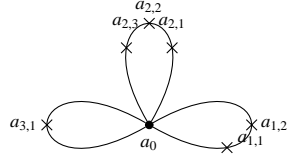
\tikzset{nxx/.style={draw, cross out, inner sep=1pt, outer sep=1pt}}
\[
 \begin{tikzpicture}
  \node[int, label=-90:{$\scriptstyle a_0$}] (v) at (0,0) {};
  \draw (v) to[in=-30, out = 30, loop, looseness=90] 
       node[nxx,midway, label=0:{$\scriptstyle a_{1,2}$}]{} 
       node[nxx,near end, label=-90:{$\scriptstyle a_{1,1}$}]{} 
       (v)
  (v) to[in=60, out = 120, loop, looseness=90]  
       node[nxx,near start, label=180:{$\scriptstyle a_{2,3}$}]{}
       node[nxx,midway, label=90:{$\scriptstyle a_{2,2}$}]{} 
       node[nxx,near end, label=0:{$\scriptstyle a_{2,1}$}]{} 
          (v)
  (v) to[in=150, out = 210, loop, looseness=90]
       node[nxx,midway, label=180:{$\scriptstyle a_{3,1}$}]{} 
       (v);
 \end{tikzpicture}
\]

Now supose that $\Gamma$ is a graph. 
Then we may similarly define a complex computing the Hochschild-Pirashvili homology on $\Gamma$ as
\[
 \HPC^\Gamma(A) = \bigotimes_{v\in V\Gamma} A \otimes \bigotimes_{e\in E\Gamma} \left(\oplus_{j\geq 0} (\overline A[1])^{\otimes j} \right).
\]
We interpret the various factors of $A$ (resp. $\overline A$) as ``sitting'' on the vertices (respectively edges) of the graph $\Gamma$. 
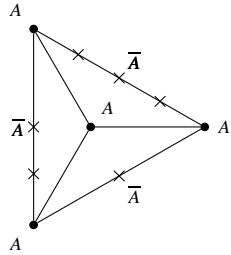
\[
 \begin{tikzpicture}
  \node[int, label=60:{$\scriptstyle A$}] (v) at (0,0) {};
  \node[int, label=0:{$\scriptstyle A$}] (w1) at (0:1.5) {};
  \node[int, label=120:{$\scriptstyle A$}] (w2) at (120:1.5) {};
  \node[int, label=-120:{$\scriptstyle A$}] (w3) at (-120:1.5) {};
  \draw (v) edge (w1) edge (w2) edge (w3) 
        (w2) to
         node[nxx,midway, label=180:{$\scriptstyle \overline A$}]{} 
         node[nxx,near end, label=180:{$\scriptstyle \overline A$}]{} 
         (w3) (w2) to
         node[nxx,midway, label=60:{$\scriptstyle \overline A$}]{} 
         node[nxx,near end, label=60:{$\scriptstyle \overline A$}]{} 
         node[nxx,near start, label=60:{$\scriptstyle \overline A$}]{} 
         (w1)
        (w1) to
        node[nxx,midway, label=-60:{$\scriptstyle \overline A$}]{} 
        (w3);
 \end{tikzpicture}
\]
The differential is then again the signed sum of contractions of spaces between decorations, multiplying the decorations accordingly.
Note that we quietly assume that an orientation of each edge and orderings of edges and vertices are chosen to make the order of the tensor products above well defined.

Let $e$ be an edge of the graph $\Gamma$, and let $\Gamma/e$ be the graph obtained by contracting $e$. 
Denote the vertices that $e$ connects by $v_1,v_2$, and the corresponding vertex of $\Gamma/e$ by $v$. Then there is a canonical map of complexes 
\begin{equation}\label{equ:HPCcontract}
 \HPC^{\Gamma}(A) \to \HPC^{\Gamma/e}(A),
\end{equation}
given by projecting the factor of the tensor product corresponding to $e$ to its ``constant'' piece
\[
 \oplus_{j\geq 0} (\overline A[1])^{\otimes j} \to (\overline A[1])^{\otimes 0}= \K,
\]
and multiplying the decorations of $v_1$ and $v_2$
\[
 \underbrace{A}_{\text{at }v_1}\otimes \underbrace{A}_{\text{at }v_2} \to \underbrace{A}_{\text{at }v}.
\]
It is not hard to check that the map \eqref{equ:HPCcontract} is a quasi-isomorphism. In fact, zig-zags of edge contractions may be used to generate the full $\Aut(F_r)$ action on $\HPC^{W_r}(A)$.

Finally, suppose that $A$ carries in addition a grading by an abelian group $G$, not necessarily equal to the cohomological $\mathbb{Z}$-grading.
Then the Hochschild-Pirashvili complexes above inherit the grading. We denote the homogeneous subcomplexes of fixed degree $g\in G$ by
\[
 \HPC^{\Gamma,g}(A),
\]
and the corresponding subspace of the homology by 
\[
 \HP^{\Gamma,g}(A),
\]

\begin{rem}\label{rem:shuffle}
 It is well known that the ordinary Hochschild complex of a (dg) commutative algebra carries a natural commutative and associative product $*_{sh}$ via the shuffle map, see \cite[section 4.2]{loday}.
 One may define a similar product on the Hochschild-Pirashvili complex, which we denote by the same letter $*_{sh}$. Concretely, on $\HPC^{\Gamma}(A)$ the product is obtained by multiplying the decorations on vertices on summing over all ways of combining the decorations on edges through shuffles.
\end{rem}

\newcommand{\FHL}{L_\infty^c}
\subsection{Our main examples}\label{sec:mainexamples}
Let $\FHL(x_1,\dots, x_N)$ be the free $L_\infty$ coalgebra cogenerated by elements $x_1,\dots,x_N$ of 
degrees $\underline d=(d_1,\dots, d_N)$.
We define $A_{\underline d}$ to be the Chevalley-Eilenberg complex of $\FHL(x_1,\dots, x_N)$, considered as an augmented  differential graded commutative algebra. Concretely, it is a symmetric algebra freely generated by $\FHL(x_1,\dots, x_N)$ shifted in degree by one.
\[
 A_{\underline d} = S(\FHL(x_1,\dots, x_N)[-1]).
\]
The dg commutative algebra $A_{\underline d}$ (and $\FHL(x_1,\dots, x_N)$) are naturally $\Z^N$-graded, according to the number of variables $x_j$ ($j=1,\dots,N$) occurring in expressions.
Finally $A_{\underline d}$ is formal, the cohomology being 
\[
 A_{\underline d}' := H(A_{\underline d}) = \K[y_1,\dots,y_N]/\langle y_iy_j=0 \rangle
 \cong \K \oplus \K y_1 \oplus \cdots \K y_N,
\]
i.e., the algebra generated by elements $y_1,\dots,y_N$ of degrees $(d_1+1,\dots,d_N+1)$ such that all non-trivial products vanish.

Let us also note that elements of $\FHL(x_1,\dots, x_N)$ can be considered graphically as linear combinations of rooted trees, with leaves labelled by numbers $1,\dots,N$. 
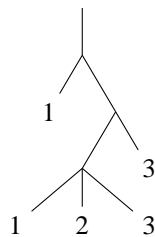
\[
\begin{tikzpicture}[sibling distance=5em, scale=.5]
  \node {}
    child {
      child { node {1} }
      child { 
        child { 
          child { node {1} }
          child { node {2} }
          child { node {3} } }
        child { node {3} } } };
\end{tikzpicture}
\]
The differential is the sum of contractions of internal edges. Similarly, elements of $A_{\underline d} $ can be considered graphically as ``bunches'' of such trees.
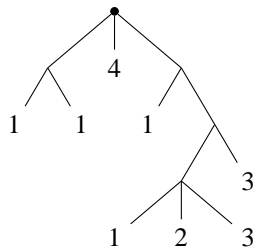
\[
\begin{tikzpicture}[sibling distance=5em, scale=.5]
  \node[int] {}
      child{ child{ node {1}} 
           child{ node {1}} }
    child{ node {4}}
    child {
      child { node {1} }
      child { 
        child { 
          child { node {1} }
          child { node {2} }
          child { node {3} } }
        child { node {3} } } };
\end{tikzpicture}
\]
The differential is again the sum of contractions of internal edges, where edges incident to the root are now considered as internal
unless they connect a root to one of the leaves.

We finally remark that combining the interpretation of $\HPC^\Gamma(A_{\underline d})$ as $A_{\underline d}$-decorations on graphs from the previous sections with the interpretation of elements of $A_{\underline d}$ as ``bunches of trees'', one sees that elements of $\HPC^\Gamma(A_{\underline d})$ may be interpreted as linear combinations of ``tree decorated graphs'' such as the following.
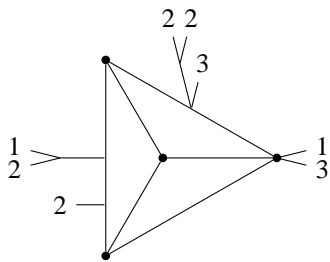
\[
 \begin{tikzpicture}[level distance=6mm, sibling distance=3mm, coo/.style={inner sep=0,outer sep=0}]
  \node[int] (v) at (0,0) {};
  \node[int] (w1) at (0:1.5) {};
  \node[int] (w2) at (120:1.5) {};
  \node[int] (w3) at (-120:1.5) {};
  \draw (v) edge (w1) edge (w2) edge (w3) 
        (w2) to
         node[midway] (x1) {}
         node[near end](x2) {} 
         (w3) (w2) to
         node[midway] (x3) {} 
         node[near end] (x4) {} 
         node[near start] (x5) {} 
         (w1)
        (w1) to
        node[midway] (x6){} 
        (w3);
  \node[coo] at (x1) {} child[grow=left] { child { node{1} } child {node {2}} };
  \node[coo] at (x2) {}  child[grow=left] { node{2} };
  \node[coo] at (w1) {}  
    [grow=right]
      child { node{3} } 
      child { node{1} } ;
   \node[coo] at (x3) {}
    [grow=up]
      child { node{3} } 
      child { child {node{2}}
              child {node{2}} } 
       ;
 \end{tikzpicture}
\]
Similarly, elements of  $\HPC^\Gamma(A'_{\underline d})$ may be interpreted as linear combinations of ``hair decorated graphs'' such as the following.
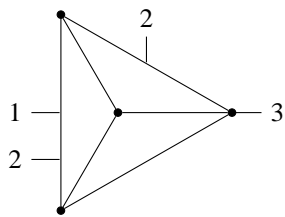
\[
 \begin{tikzpicture}[level distance=6mm, sibling distance=3mm, coo/.style={inner sep=0,outer sep=0}]
  \node[int] (v) at (0,0) {};
  \node[int] (w1) at (0:1.5) {};
  \node[int] (w2) at (120:1.5) {};
  \node[int] (w3) at (-120:1.5) {};
  \draw (v) edge (w1) edge (w2) edge (w3) 
        (w2) to
         node[midway] (x1) {}
         node[near end](x2) {} 
         (w3) (w2) to
         node[midway] (x3) {} 
         node[near end] (x4) {} 
         node[near start] (x5) {} 
         (w1)
        (w1) to
        node[midway] (x6){} 
        (w3);
  \node[coo] at (x1) {}  child[grow=left] { node{1} };
  \node[coo] at (x2) {}  child[grow=left] { node{2} };
  \node[coo] at (w1) {}  child[grow=right] { node{3} };
  \node[coo] at (x3) {}  child[grow=up] { node{2} };
 \end{tikzpicture}
\]
Thanks to the projection $A_{\underline d}\to A_{\underline d}'$, one has a  quasi-isomorphism $\HPC^\Gamma(A_{\underline d})\to \HPC^\Gamma(A'_{\underline d})$.

\subsection{Recollection: Hochschild-Pirashvili homology of $A_d$}\label{sec:HPofAd}
Let us specialize further to the case $N=1$, i.e., of only one generator, which we denote by $x$. In this case we write $A_d=A_{(d)}$ for short. The Hochschild-Pirashvili homology of $A_d$ (or, equivalently $A_d'$) on $W_r$ (or, equivalently, any graph $\Gamma$ of loop order $r$) has been computed in \cite[section 2]{TW2}. Below we cite some facts from loc. cit.

\subsubsection{$d$ even}\label{sss331}
In this case  $\FHL(x)$ is one-dimensional, and one has $A'_d=A_d$ is 2-dimensional. For any connected graph $\Gamma$ of loop order $r$ denote by $C_\bullet(\Gamma)$ its cellular chain complex
\[
0\to C_1(\Gamma) \to C_0(\Gamma)\to 0,
\]
which according to our cohomological conventions is concentrated in degrees $-1$ and $0$.
  For the degree $d$ of $x$ even and $h\in \Z$ we may identify the Hochschild-Pirashvili complex with a symmetric power:
 \beq{equ:HPCisSC}
  \HPC^{\Gamma,h}(A_d) = S^h\left(C_\bullet(\Gamma)[-d-1]\right),
 \eeq
where the superscript $h$ refers to the piece of auxiliary degree $h$, assigning $x$ degree $1$. Indeed, since $A'_d=A_d$, only
hairs can appear as decorations. Elements of $C_0(\Gamma)$ have an odd degree $d+1$ and thus can not be repeated in the symmetric power, which correspond to the fact that only one hair can grow from a vertex.
On the other hand,
 hairs on an edge can be repeated, which  correspond to the fact that  in the symmetric power elements of $C_1(\Gamma)$ have even degree $d$ and thus can also be taken with any multiplicity. 
   It follows immediately that 
 \[
  \HP^{\Gamma,h}(A_d) = 
  S^h(H_\bullet(\Gamma)[-d-1])
  = 
  S^h(H_1(\Gamma)[-d]) \oplus H_0(\Gamma)[-d-1]\otimes S^{h-1}(H_1(\Gamma)[-d]).
 \]
Of course, in our example $H_0(\Gamma)=\K$ and $H_1(\Gamma) = \K^r$, but we will keep the above notation for reasons apparent later.
 The action of $\Out(F_r)$ in this case factors through the obvious action of $\GL(r,\mathbb{Z})$ on $H_1(\Gamma)\cong \K^r$.
  
 \subsubsection{$d$ odd}\label{sss332}
 If $d$ is odd, then $\FHL(x)$ is infinite-dimensional, however with two-dimensional cohomology given by the cofree Lie coalgebra cogenerated by $x$. Let $\Gamma$ be a finite connected graph of loop order $r$. In case of odd $d$ neither the complex $\HPC^\Gamma(A_d)$, nor $\HPC^\Gamma(A'_d)$
 can be expressed via $C_\bullet(\Gamma)$. In order to compute their cohomology we  consider the special case $\HPC^{W_r}(A'_d)$.
  Its component $\HPC^{W_r,h}(A'_d)$ is a two-term complex that may be identified with the de Rham map
 on an $n$-dimensional odd vector space
 \begin{multline}\label{eq:derham}
  0\to \bigoplus_{l+2l'=h} S^l(H_1(W_r)[-d]) \otimes S^{l'}(H_1(W_r)[-2d])  
  \xrightarrow{\rm d_{dR}} \\
  \bigoplus_{l+2l'=h-1} H_0(W_r)[-d-1] \otimes S^l(H_1(W_r)[-d]) \otimes S^{l'}(H_1(W_r)[-2d]) \to 0.
 \end{multline}
 If we denote a basis of $H_1(W_r)[-d]$ by variables $y_j$ of degree $d$ and a basis of $H_1(W_r)[-2d]$ by variables $\theta_j$ of degree $2d$, and if we interpret the spaces in the complex above as polynomials in $y_j$ and $\theta_j$, then the ``de Rham'' differential $\rm d_{dR}$ is given by the formula
 \beq{equ:dRDef}
  \rm d_{dR}f(y_1,\dots,y_n,\theta_1,\dots,\theta_n)= \sum_{j=1}^n y_j \frac\partial {\partial \theta_j} f(y_1,\dots,y_n,\theta_1,\dots,\theta_n).
 \eeq
 The first term of this complex is spanned by graphs $W_r$ that have only edges decorated by hairs, the second term being spanned by 
 similar graphs whith the only vertex of $W_r$ decorated by a hair. A monomial $\theta_j^k$ should be understood as the $j$-th circle
 having $2k$ hairs; and a monomial $y_j\theta_j^k$ should be understood as the $j$-th circle having $2k+1$ hairs.
 
 We will set
 \begin{align*}
  B_{r,d}^{h,I} &= \ker({\rm d}_{dR})
  &
  B_{r,d}^{h,II} &= \coker({\rm d}_{dR}).
 \end{align*}
Both $B_{r,d}^{h,I}$ and $B_{r,d}^{h,II}$ are $\Out(F_r)$-modules via the $\Out(F_r)$-action on the Hochschild-Pirashvili homology. Note that there is also a natural action of $\GL(r,\Z)$ on these spaces. However, the $\Out(F_r)$-action does not in general  factor through the $\GL(r,\Z)$-action.
There is a filtration (called the Hodge filtration) on the above modules, such that the induced $\Out(F_r)$-action on the associated graded spaces factors through $\GL(r,\Z)$, see \cite{TW2} for details. In fact the obvious $\GL(r,\Z)$ action that one can see is  the action on the associated graded spaces.

Furthermore one can see that $B_{r,d}^{h,I}\cong B_{r,d}^{h+2,II}[2d+1]$ since the kernel and cokernel of the de Rham differential can be identified. Concretely, an isomorphism $B_{r,d}^{h+2,II}\to B_{r,d}^{h,I}[-2d-1]$ is given by the operator \eqref{equ:dRDef}.
Finally, the odd number $d$ matters only in providing a global degree shift. We will set 
\[
 B_r^h := B_{r,1}^{h,I}[h]
\]
so that $B_r^h$ is concentrated in degree zero. Thus one has
\[
 B_{r,d}^{h,I} \cong B_r^h[-hd].
\]

\section{Outer space and ``decorated'' graph complexes}\label{sec:outer}

\subsection{Outer space and coefficient systems}
A metrized graph is a combinatorial graph together with the assignment of a non-negative number (the ``length'') to each edge, so that the sum of the lengths of the edges of any closed loop is positive.\footnote{We allow in particular graphs with short loops, i.e., edges connecting a vertex to itself.}
The (quotient of the) outer space $\OS_r$ defined by Culler and Vogtmann \cite{CV} is the quotient of the space of isomorphism classes of connected metrized $r$-loop graphs, obtained by identifying a metrized graph with the graph obtained by contracting all edges of zero lengths. 
The space $\OS_r$ is naturally an open orbi-cell complex. Its (open) cell structure is neatly encoded in the following category.
Define the objects of the category $\Gr_r$ to be $r$-loop graphs (not isomorphisms classes of graphs), and the morphisms to be generated by the following maps of graphs:
\begin{itemize}
 \item Isomorphisms of graphs.
 \item Sub-forest contractions.
\end{itemize}
Then the $d$-dimensional orbicells of $\OS_r$ correspond to isomorphism classes of $\Gr_r$ represented by graphs with $d$ edges. The attachment maps between orbicells are given by the arrows of $\Gr_r$ between these isomorphism classes.

The space $\OS_r$ plays a role similar to the classifying space of the group of outer automorphisms of the free group on $r$ generators $F_r$. More concretely, one may define a space $\OX_r$ whose points are isomorphism classes of metrized graphs, with a homotopy class of a map from a wedge product of $r$ circles, inducing an isomorphism on $\pi_1$.
One again identifies a metrized graph with the graph obtained by contracting all edges of length 0.
By precomposing the map from the wedge of circles accordingly, the group $\Out(F_r)$ acts on $\OX_r$. The stabilizer subgroup may be identified with the automorphism group of the metrized graph and is hence finite.

The space $\OX_r$ has an open cell complex structure, the cells corresponding to isomorphism classes of (non-metrized) graphs. 
We may again organize the cells into a category $\Grx_r$ whose objects are graphs with a homotopy class of a map from $\vee_r S^1$, inducing an isomorphism on $\pi_1$.
The morphisms are generated by isomorphisms of graphs and sub-forest contractions.
Again we have an action of $\Out(F_r)$ on $\Grx_r$ by pre-composing the maps from $\vee_r S^1$ accordingly.
Finally we may recover $\Gr_r$ as the quotient of $\Grx_r$, and denote the forgetful functor by
\begin{equation}\label{equ:Grproj}
\pi \colon \Grx_r \to  \Gr_r=\Grx_r/\Out(F_r).
\end{equation}

We are interested in studying local systems on outer space. The corresponding data may be nicely organized using the category $\Gr_r$.

\begin{defi}\label{def:locsys}
We define a coefficient system as a functor
\[
\Gr_r\to dg\Vect.
\]
We say that the coefficient system is a homotopy local system if the images of all arrows are  quasi-isomorphisms.
We say it is a local system if the images of all arrows are isomorphisms.
\end{defi}

We are interested  mainly in the following examples.

\begin{ex}\label{ex:Hcoeff}
 The functors 
 \begin{align*}
  C_\bullet: \Gr_r &\to dg\Vect  & H_\bullet: \Gr_r &\to dg\Vect 
  \\
  \Gamma &\mapsto C_\bullet(\Gamma)  & \Gamma &\mapsto H_\bullet(\Gamma)
 \end{align*}
 assigning to a graph its cellular chain complex (respectively its homology) is a coefficient system.
\end{ex}

\begin{ex}\label{ex:Detcoeff}
 The functor 
 \begin{align*}
  \Det: \Gr_r \to \Vect 
  \\
  \Gamma \mapsto \wedge^{r}H_1(\Gamma)
 \end{align*}
 defines a (one dimensional) coefficient system.
\end{ex}

\begin{ex}\label{ex:HPCcoeff}
Let $A$ be a (dg) commutative algebra.
Define the functor 
\[
\HP_A \colon \Gr_r \to g\Vect
\]
which assigns to a graph $\Gamma$ the Hochschild-Pirashvili homology $\HP^\Gamma(A)$.
Similarly, define the functor 
\[
\HPC_A \colon \Gr_r \to dg\Vect
\]
assigning to $\Gamma$ the Hochschild-Pirashvili complex $\HPC^\Gamma(A)$.
In particular, $\HP_A$ is the composite of $\HPC_A$ and the homology functor from $dg\Vect$ to $g\Vect$.
In case $A$ is $G$-graded, for any $g\in G$ we similarly define the functors $\HP_A^g$, $\HPC_A^g$ assigning to $\Gamma$
the $g$-part $\HP^{\Gamma,g}(A)$, respectively, $\HPC^{\Gamma,g}(A)$, of the higher Hochschild homology/complex.
\end{ex}

\begin{ex}
More generally, let $V$ be any representation of $\Out(F_r)$. Then a  local system can be 
constructed by sending a graph $\Gamma\in \ob\Gr_r$ to the vector space
\[
 \left( \bigoplus_{\Gamma'\in \ob\,\pi^{-1}(\Gamma)} V\right)_{\Out(F_r)}
\]
where $\pi$ is as in \eqref{equ:Grproj} $\Out(F_r)$ acts on $V$ and simultaneously by permuting the summands according to the action on $\pi^{-1}(\Gamma)$.

Conversely, given a local system on outer space as in the above definition, one may recover a representation of $\Out(F_r)$ as follows: First, one defines the fundamental group $\pi_1(\Gr_r)$ as the group of homotopy classes of zigzags of arrows 
\[
 W_r \to \cdot \leftarrow W_r
\]
where $W_r$ is a wedge of $r$ circles, considered as a graph. Two such zigzags are homotopic if the ``space between them'' may be triangulated by commutative diagrams, cf. the similar Definition \cite[III.C Definition 3.5]{BH}.
Clearly, any local system yields a representation of $\pi_1(\Gr_r)$, and one can check that $\pi_1(\Gr_r)=\Out(F_r)$.
\end{ex}
%

Finally, let us note that one can (of course) define direct sums and tensor products of local systems in the obvious way.

\subsection{Decorated graph complexes}
Suppose we are given a coefficient system
\[
 F \colon \Gr_r \to dg\Vect
\]
as above. Given a graph $\Gamma$, let us denote the full subcategory of $\Gr_r$ whose objects are isomorphic to $\Gamma$ by $\Gr_r^{[\Gamma]}$.
Of course the functor $F$ restricts to a functor 
\[
 F\mid_{\Gr_r^{[\Gamma]}} \colon \Gr_r^{[\Gamma]} \to dg\Vect.
\]
Furthermore, we consider the functor 
\[
 \Or \colon \Gr_r^{[\Gamma]} \to g\Vect
\]
assigning to a graph $\Gamma$ the one-dimensional graded vector space 
\[
 \Or(\Gamma) = (\K[1])^{\otimes |E\Gamma|} ,
\]
where $E\Gamma$ is the edge set of $\Gamma$. An isomorphism $f:\Gamma\to \Gamma'$ is sent by $\Or$ to the map of one-dimensional vector spaces obtained by permuting the factors in the tensor product according to the edge permutation induced by $f$. (Concretely, choosing a basis $\Or(f)$ acts by multiplication by $\pm 1$.) This functors encodes  possible choices for the orientation of the orbicells 
in $\OS_r$.

Then define the vector space
\[
 V_{[\Gamma]} = \colim\left( \Or\otimes F\mid_{\Gr_r^{[\Gamma]}} \right)
\]
as the colimit of the corresponding sub-diagram of $\Or \otimes F$. Note that $\Gr_r^{[\Gamma]}$ is a connected groupoid, thus this colimit is isomorphic to $\left(\Or(\Gamma)\otimes F(\Gamma)\right)_{G_\Gamma}$, where $G_\Gamma$ is the group of symmetries of $\Gamma$.

Then we define a graph complex $\G^r_{F}$ as follows. As a graded vector space
\[
 \G^r_{F} = \oplus_{[\Gamma]} V_{[\Gamma]},
\]
where the sum ranges over isomorphism  classes of $r$-loop graphs.
The differential is defined as 
\beq{equ:diffGCdeco}
 d (\Gamma, \varepsilon\otimes v) =(\Gamma, \varepsilon\otimes d_{F(\Gamma)} v) +
  \sum_{e\in E\Gamma} \bigl(c_e(\Gamma),  \Or(c_e)(\varepsilon) \otimes F(c_e)(v)\bigr),
\eeq
where the sum ranges over edges of the graph $\Gamma$, $c_e$ is the morphism in $\Gr_r$ contracting the edge $e$, and 
\[
 \Or(c_e) : \Or(\Gamma)= (\K[1])^{\otimes |E\Gamma|} \to \Or(c_e(\Gamma))= (\K[1])^{\otimes |Ec_e(\Gamma)|}
\]
is the morphism of degree one of one-dimensional vector spaces corresponding to removing the factor corresponding to the edge $e$. (Concretely, picking a basis, $\Or(c_e)$ acts by multiplication by an alternating sign as $e$ ranges over the edges of $\Gamma$.)

Finally, we define the dual graph complex 
\[
 \GC^r_{F} = (\G^r_{F})^*.
\]

\begin{rem}\label{r:sheaf}
A functor $F\colon \Gr_r\to dg\Vect$ is what is called a  {\it cellular orbi-cosheaf} on $\OS_r$. The complex $ \G^r_{F}$  computes the {\it Borel-Moore homology} of $\OS_r$ with coefficients in $F$~\cite[Section~6.2]{Curry}, \cite{LazVor}. The dual graph complex $\GC^r_F$
computes the {\it locally compact cohomology} of $\OS_r$ with coefficients in the {\it cellular orbi-sheaf}  $F^*\colon \Gr_r^{op}\to dg\Vect$ --
the objectwise dual of $F$.
\end{rem}

\subsection{Hairy graph complexes -- Standard definition}\label{sec:hairydef_std}
The hairy graph complexes $\HGC_{m,n}$ are combinatorial complexes of formal series of graphs with vertices of valence $1$ (hairs) or valence $\geq 3$ (internal vertices) as depicted in \eqref{equ:HGCsample}. They arise in the study of the deformation theory of the $E_n$ operads
that we briefly review in Section~\ref{s:operads}.

Each complex $\HGC_{m,n}$ is spanned by  connected graphs having some set of {\it external vertices} of valence~1, and some  set of {\it internal vertices}. The set of external vertices, which are also called hairs, must be non-empty.  We also assume that all internal vertices are of valence~$\geq 3$. We allow graphs  to have  multiple edges and loops (edges joining a vertex to itself). For such a graph define its {\it orientation set} as the union of the set of its external vertices (considered as elements of degree $m$), the set of its internal vertices (considered as elements of degree $n$), and the set of its edges (considered as elements of degree $(1-n)$). By an {\it orientation} of a graph we will understand ordering of its orientation set together with an orientation of all its edges. Two such graphs are {\it equivalent} if there is a bijection between their sets of vertices and edges respecting the adjacency structure of the graphs, orientation of the edges, and the order of the orientation sets. The space of $\HGC_{m,n}$ is the quotient space of the vector space freely spanned by such graphs modulo the orientation relations:

\vspace{.2cm}

(1) $\Gamma_1=(-1)^n\Gamma_2$ if $\Gamma_1$ differs from $\Gamma_2$ by an orientation of an edge.

(2) $\Gamma_1=\pm\Gamma_2$, where $\Gamma_2$ is obtained from $\Gamma_1$ by a permutation of the orientation set. The sign here is the Koszul sign of permutation.

\vspace{.2cm}

The differential $\partial\Gamma$ of a graph $\Gamma\in\HGC_{m,n}$ is defined as the sum of expansions of its internal vertices.  The orientation set of a new graph is obtained by adding the new vertex and the new edge as the first and second elements to the orientation set, and by orienting the new edge from the old vertex to the new one.  
 We define the degree of a graph as the sum of degrees of the elements from its orientation set minus $m$. We need this shift by $m$ 
 to unable the complex $\HGC_{m,n}$ with the dg Lie algebra structure, which is related to the Browder operator in the rational homotopy of the spaces of long embeddings modulo immersions and also appears naturally in the operad deformation theory~\cite{FTW,TW}.

Combinatorially, the Lie bracket connects a hair of one graph to an internal vertex of another graph in all possible ways as indicated in the following picture.
\beq{equ:grbracket}
\left[ 
\begin{tikzpicture}[baseline=-.8ex]
\node[draw,circle] (v) at (0,.3) {$\Gamma$};
\draw (v) edge +(-.5,-.7) edge +(0,-.7) edge +(.5,-.7);
\end{tikzpicture}
,
\begin{tikzpicture}[baseline=-.65ex]
\node[draw,circle] (v) at (0,.3) {$\Gamma'$};
\draw (v) edge +(-.5,-.7) edge +(0,-.7) edge +(.5,-.7);
\end{tikzpicture}
\right]
=
\sum
\begin{tikzpicture}[baseline=-.8ex]
\node[draw,circle] (v) at (0,1) {$\Gamma$};
\node[draw,circle] (w) at (.8,.3) {$\Gamma'$};
\draw (v) edge +(-.5,-.7) edge +(0,-.7) edge (w);
\draw (w) edge +(-.5,-.7) edge +(0,-.7) edge +(.5,-.7);
\end{tikzpicture}
\mp 
\sum
\begin{tikzpicture}[baseline=-.8ex]
\node[draw,circle] (v) at (0,1) {$\Gamma'$};
\node[draw,circle] (w) at (.8,.3) {$\Gamma$};
\draw (v) edge +(-.5,-.7) edge +(0,-.7) edge (w);
\draw (w) edge +(-.5,-.7) edge +(0,-.7) edge +(.5,-.7);
\end{tikzpicture}\, .
\eeq


\subsection{Hairy graph complexes as decorated graph complexes}
The hairy graph complexes $\HGC_{m,n}$ from the previous section split into direct products of subcomplexes $\HGC_{m,n}^{r,h}$ of fixed loop number $r$ and fixed number of hairs $h$, so that 
\[
 \HGC_{m,n} = \prod_{r\geq 0 \atop {h\geq 1}} \HGC_{m,n}^{r,h}.
\]

One easy but crucial observation is that the pieces $\HGC_{m,n}^{r,h}$, $r\geq 2$, $h\geq 1$, may be equally well defined as decorated graph complexes.
More concretely, for $n-m$ even note that hairs are odd objects, and hence there are no vertices with multiple hairs in graphs in $\HGC_{m,n}^{r,h}$. For $r\geq 2$ each hairy graph $\Gamma \in\HGC_{m,n}^{r,h}$ thus consists of a ``core graph'' of vertices that have at least 3 non-hair neighbors, connected by strings of hairs as the following picture illustrates.
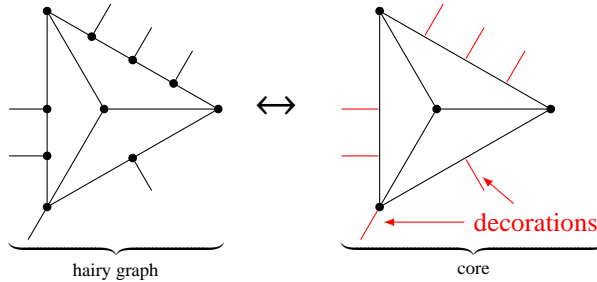
\[
\underbrace{
 \begin{tikzpicture}[level distance=6mm, sibling distance=3mm, coo/.style={inner sep=0,outer sep=0}, baseline=-.65ex]
  \node[int] (v) at (0,0) {};
  \node[int] (w1) at (0:1.5) {};
  \node[int] (w2) at (120:1.5) {};
  \node[int] (w3) at (-120:1.5) {};
  \draw (v) edge (w1) edge (w2) edge (w3) 
        (w2) to
         node[midway, int] (x1) {}
         node[near end, int](x2) {} 
         (w3) (w2) to
         node[midway, int] (x3) {} 
         node[near end, int] (x4) {} 
         node[near start, int] (x5) {} 
         (w1)
        (w1) to
        node[midway, int] (x6){} 
        (w3)
        (x3) edge +(60:.5)
        (x4) edge +(60:.5)
        (x5) edge +(60:.5)
        (x1) edge +(180:.5)
        (x2) edge +(180:.5)
        (x6) edge +(-60:.5)
        (w3) edge +(-120:.5);
        ;
 \end{tikzpicture}
 }_{\text{hairy graph}}
\quad \mbox{\huge $\leftrightarrow$ }\quad
\underbrace{
 \begin{tikzpicture}[level distance=6mm, sibling distance=3mm, coo/.style={inner sep=0,outer sep=0}, baseline=-.65ex]
  \node[int] (v) at (0,0) {};
  \node[int] (w1) at (0:1.5) {};
  \node[int] (w2) at (120:1.5) {};
  \node[int] (w3) at (-120:1.5) {};
  \draw (v) edge (w1) edge (w2) edge (w3) 
        (w2) to
          node[midway,coo] (x1) {}
          node[near end,coo](x2) {} 
         (w3) (w2) to
          node[midway,coo] (x3) {} 
          node[near end,coo] (x4) {} 
          node[near start,coo] (x5) {} 
         (w1)
        (w1) to
         node[midway,coo] (x6){} 
        (w3)
        ;
\draw[red] (x3) edge +(60:.5)
        (x4) edge +(60:.5)
        (x5) edge +(60:.5)
        (x1) edge +(180:.5)
        (x2) edge +(180:.5)
        (x6) edge +(-60:.5)
        (w3) edge +(-120:.5);
\node[red] (t) at (1.3,-1.5) {decorations};
\draw[red,-latex] (t) edge (0.6,-.9) edge (-.7,-1.5);
 \end{tikzpicture}
 }_{\text{core}}
\]
Comparing this to the graphical interpretation for the  Hochschild-Pirashvili complex of a graph from sections \ref{sec:moreexplicit} and \ref{sec:mainexamples}, we see that the complex $\HGC_{m,n}^{r,h}$ essentially agrees with the complex $\GC^r_F$ for $F$ the coefficient system given by the Hochschild-Pirashvili complex as in Example \ref{ex:HPCcoeff}. More precisely, taking into account the degrees and signs we find that if $n$, $m$ are even and $r\geq 2$ then
\beq{equ:HGCdecoee}
 \HGC_{m,n}^{r,h} = \GC^r_{\HPC^h_{A_{n-m-2}}}[nr+m-n].
\eeq
If $n$ and $m$ are odd, then additional signs in the definition of $\HGC_{m,n}$ yield the identification
\beq{equ:HGCdecooo}
 \HGC_{m,n}^{r,h} = \GC^r_{\HPC^h_{A_{n-m-2}}\otimes \Det}[nr+m-n].
\eeq
where $\Det$ is the coefficient system from example \ref{ex:Detcoeff}.

For $n-m$ odd the hairs on hairy graphs are even objects. In particular, hairy graphs $\Gamma\in \HGC_{m,n}^{r,h}$ for $r\geq 2$ may contain multiple hairs at a vertex, and also tree-like ``antennas'' as in the following example.
Still, recursively cutting the antennas we may view each hairy graph $\Gamma\in \HGC_{m,n}$ as a ``core'' graph, whose edges are decorated by strings of bunches of ``antennas'', as the following picture shall indicate.
 \[
 \underbrace
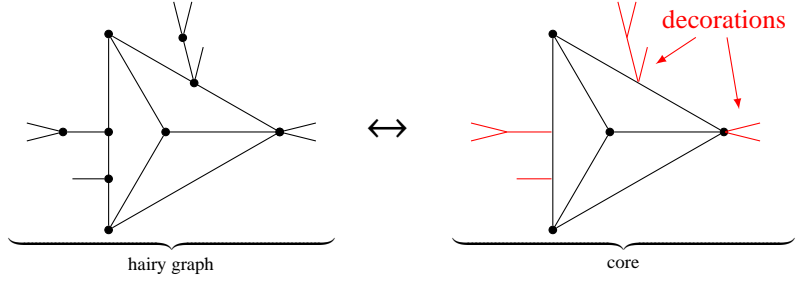
{
 \begin{tikzpicture}[level distance=6mm, sibling distance=3mm, coo/.style={inner sep=0,outer sep=0}, baseline=-.65ex]
  \node[int] (v) at (0,0) {};
  \node[int] (w1) at (0:1.5) {};
  \node[int] (w2) at (120:1.5) {};
  \node[int] (w3) at (-120:1.5) {};
  \draw (v) edge (w1) edge (w2) edge (w3) 
        (w2) to
         node[midway] (x1) {}
         node[near end](x2) {} 
         (w3) (w2) to
         node[midway] (x3) {} 
         node[near end] (x4) {} 
         node[near start] (x5) {} 
         (w1)
        (w1) to
        node[midway] (x6){} 
        (w3);
  \node[int] at (x1) {} child[grow=left] { node[int]{} child { node{} } child {node {}} };
  \node[int] at (x2) {}  child[grow=left] { node{} };
  \node[int] at (w1) {}  
    [grow=right]
      child { node{} } 
      child { node{} } ;
   \node[int] at (x3) {}
    [grow=up]
      child { node{} } 
      child { node[int]{} child {node{}}
              child {node{}} } 
       ;
 \end{tikzpicture}
  }_{\text{hairy graph}}
\quad \mbox{\huge $\leftrightarrow$ }\quad
\underbrace{
 \begin{tikzpicture}[level distance=6mm, sibling distance=3mm, coo/.style={inner sep=0,outer sep=0}, baseline=-.65ex]
  \node[int] (v) at (0,0) {};
  \node[int] (w1) at (0:1.5) {};
  \node[int] (w2) at (120:1.5) {};
  \node[int] (w3) at (-120:1.5) {};
  \draw (v) edge (w1) edge (w2) edge (w3) 
        (w2) to
         node[midway] (x1) {}
         node[near end](x2) {} 
         (w3) (w2) to
         node[midway] (x3) {} 
         node[near end] (x4) {} 
         node[near start] (x5) {} 
         (w1)
        (w1) to
        node[midway] (x6){} 
        (w3);
  \node[coo] at (x1) {} [red] child[grow=left] { child { node{} } child {node {}} };
  \node[coo] at (x2) {} [red] child[grow=left] { node{} };
  \node[coo] at (w1) {}  
    [grow=right, red]
      child { node{} } 
      child { node{} } ;
   \node[coo] at (x3) {}
    [grow=up, red]
      child { node{} } 
      child { child {node{}}
              child {node{}} } 
       ;
       \node[red] (t) at (1.5,1.5) {decorations};
\draw[red,-latex] (t) edge (0.6,.9) edge (1.7,.3);
 \end{tikzpicture}
 }_{\text{core}}
 \]
Comparing this again to the definition of the Hochschild-Pirashvili complex for graphs from sections \ref{sec:moreexplicit} and \ref{sec:mainexamples}, we can readily identify 
\beq{equ:HGCdecooe}
 \HGC_{m,n}^{r,h} = \GC^r_{\HPC^h_{A_{n-m-2}}}[nr+m-n]
\eeq
for $n$ even, $m$ odd, $r\geq 2$ and 
\beq{equ:HGCdecoeo}
 \HGC_{m,n}^{r,h} = \GC^r_{\HPC^h_{A_{n-m-2}}\otimes \Det}[nr+m-n]
\eeq
for $n$ odd, $m$ even and $r\geq 2$.

For the purposes of this paper, the reader may take \eqref{equ:HGCdecoee}-\eqref{equ:HGCdecoeo} as the primary definitions of the hairy graph complexes. The only downside is that the Lie bracket on $\HGC_{m,n}$ is not readily visible using this definition.

Furthermore, the above definitions do not capture the loop orders $r=0$ and $r=1$. 
Fortunately, in loop order $\leq 1$ the hairy graph cohomology is known, with the following result~\cite[Proposition 3.3]{Turchin3}.
\begin{thm}\label{thm:loop01}
 The zero-loop and one-loop pieces $H(\HGC_{m,n}^0)$ and $H(\HGC_{m,n}^1)$ of the hairy graph cohomology satisfy
\[
 H(\HGC_{m,n}^0) =
 \begin{cases}
  \K \, \linegraph & \text{for $n-m$ even} \\
  \K \, \tripodgraph & \text{for $n-m$ odd},
 \end{cases}
\]
and
\[
 H(\HGC_{m,n}^1) = 
 \prod_{k \geq 1 \atop{k\equiv Ln+1 \text{ mod }2L } }
 \K 
\begin{tikzpicture}[baseline=-.65ex, scale=.7]
\node[int] (v1) at (0:1) {};
\node[int] (v2) at (72:1) {};
\node[int] (v3) at (144:1) {};
\node[int] (v4) at (216:1) {};
\node (v5) at (-72:1) {$\cdots$};
\draw (v1) edge (v2) edge (v5) (v3) edge (v2) edge (v4) (v4) edge (v5);
\draw (v1)  edge +(0:.6)  ; 
\draw (v2) edge +(72:.6)  ; 
\draw (v3) edge +(144:.6) ; 
\draw (v4)  edge +(226:.6)  ; 
\end{tikzpicture}
\quad(\text{$k$ vertices}),
\]
where $L=1$ if $n-m$ is even, and $L=2$ if $n-m$ is odd.
\end{thm}

\section{Proof of Theorem~\ref{thm:main_even}}\label{s:th1}
Applying~\eqref{equ:HGCdecoee}-\eqref{equ:HGCdecooo} to the equivalence of coefficient systems~\eqref{equ:HPCisSC} in case of even codimension $n-m$, we express the hairy graph-complexes as follows
\beq{eq:hgc_gc}
\HGC_{m,n}^{h,r}\cong \GC^r_{\Det^{\otimes n}\otimes S^h(C_\bullet[-n+m+1])}[nr+m-n],
\eeq
where $C_\bullet\colon \Gr_r\to dg\Vect$ is the coefficient system from Example~\ref{ex:Hcoeff}.

Denote by $\tilde C_\bullet$ the coefficient system assigning to   $\Gamma\in\Gr_r$ the kernel of the augmentation map
\beq{eq:pi}
\pi\colon C_\bullet(\Gamma)\to\K\cong H_0(\Gamma).
\eeq

\begin{lemma}\label{l:splitting}
There is a natural isomorphism of coefficient systems (i.e., functors $\Gr_r\to dg\Vect$)
\[
\K\oplus \tilde C_\bullet\cong C_\bullet.
\]
\end{lemma}
\begin{proof}
In order to show this splitting we have to construct a section $s\colon \K\to C_\bullet$ to the projection $\pi\colon C_\bullet\to \K$.
On a graph $\Gamma$ this section 
\[
H_0(\Gamma)\cong \K \to C_0(\Gamma)
\]
sends the generator 1 to 
\[
\frac 1{2r-2}\sum_{v\in V\Gamma} (\mathit{val}(v)-2)\cdot v.
\]
Here $\mathit{val}(v)$ stays for the valence of a vertex $v$, and $V\Gamma$ (below, respectively, $E\Gamma$) is the set of vertices (respectively, edges) of~$\Gamma$. One has that the sum of coefficients $\sum_{v\in V\Gamma} (\mathit{val}(v)-2)=2|E\Gamma|-2|V\Gamma|$ which is minus double Euler characteristic of $\Gamma$, i.e. $2r-2$. The additional prefactor $(\mathit{val}(v)-2)$ makes the transformation natural in $\Gamma$. Concretely, the contraction of an edge between vertices $v$ and $v'$ produces a vertex $w$ such that
\[
 \mathit{val}(v)-2 = (\mathit{val}(v)-2) + (\mathit{val}(v')-2).
\]
\end{proof}

\begin{cor}\label{cor:formal}
The coefficient system $C_\bullet\colon \Gr_r\to dg\Vect$ is formal, i.e.  there is a natural quasi-isomorphism:
\[
H_\bullet \stackrel{\sim}{\Rightarrow} C_\bullet.
\]
\end{cor}
\begin{proof}
On a graph $\Gamma$ the natural transformation is defined as follows.
The map 
\[
H_1(\Gamma) = \ker(C_1(\Gamma) \to C_0(\Gamma)) \hookrightarrow C_1(\Gamma)
\]
is the natural inclusion. The map $s\colon
H_0(\Gamma)\cong \K \to C_0(\Gamma)$ is as defined in the proof of the lemma above. 
\end{proof}

This formality and essentially equivalent to it the splitting from Lemma~\ref{l:splitting} is the main reason for the splitting in the hairy graph-homology. Indeed, the splitting of the lemma implies the splitting of the coefficient system:
\begin{multline}\label{eq:coef_split}
S^h(C_\bullet[-n+m+1])\cong S^h(\K[-n+m+1]\oplus \tilde C_\bullet[-n+m+1])\cong\\
S^h(\tilde C_\bullet[-n+m+1])\oplus S^{h-1}(\tilde C_\bullet[-n+m+1])[-n+m+1].
\end{multline}
The complexes $\HGC_{m,n}^{h,r,I}$ and $\HGC_{m,n}^{h,r,II}$ are defined as the direct summands of $\HGC_{m,n}^{h,r}$ arising through this splitting into these two coefficient systems and isomorphism~\eqref{eq:hgc_gc}. Since these coefficient systems are formal, 
one has
\begin{multline*}
H(\HGC_{m,n}^{h,r,I}) = H(\GC^r_{\Det^{\otimes n}\otimes S^h(\tilde C_\bullet[-n+m+1])})[nr+m-n] = \\
=H(\GC^r_{\Det^{\otimes n}\otimes S^h(H_1[-n+m+2])})[nr+m-n]= 
H(\GC^r_{\Det^{\otimes n}\otimes S^hH_1})[nr+(h-1)(n-m-2)-2].
\end{multline*}
And similarly for the second summand,
\[
H(\HGC_{m,n}^{h,r,II})= H(\GC^r_{\Det^{\otimes n}\otimes S^{h-1}H_1})[nr+(h-1)(n-m-2)-1].
\]

The map $\pi\colon C_\bullet\to\K$ can be extended as derivation to the symmetric power of $C_\bullet[-n+m+1]$ defining a
map of coefficient systems
\beq{eq:D_pi}
D_\pi\colon S^{h+1} (C_\bullet[-n+m+1])\Rightarrow S^h (C_\bullet[-n+m+1])[-n+m+1].
\eeq

\begin{prop}\label{P:L}
The map $\frac 12[L,-]\colon \HGC_{m,n}^{h,r}\to \HGC_{m,n}^{h+1,r}[-n+m+1]$ in view of identification~\eqref{eq:hgc_gc} 
is described as the map
\[
\GC^r_{id\otimes D_\pi}\colon \GC^r_{\Det^{\otimes n}\otimes S^h (C_\bullet[-n+m+1])[-n+m+1]}
\to \GC^r_{\Det^{\otimes n}\otimes S^{h+1} (C_\bullet[-n+m+1])}
\]
induced by the map $D_\pi$ of coefficient systems.
\end{prop}

\begin{proof}
Indeed, the Lie bracket operation $[L,-]$ on graphs is combinatorially the operation of adding one additional hair at a vertex, i.e.,
\[
 [L, \Gamma] = 2\sum_{v\in V\Gamma} \Gamma \cup (\text{hair at $v$}).
\]
The coefficient 2 appears since each of two vertices of $L$ contributes. Dually, this operation is the sum of cuttings off hairs from one of the vertices, which in terms of maps of coefficient systems is exactly $D_\pi$.
\end{proof}

To finish the proof of the theorem we notice that the map $D_\pi$ of coefficients is compatible with the splitting~\eqref{eq:coef_split}.
\hfill \qed

\section{Proof of Theorem~\ref{thm:main_odd}}\label{sec:proofs}
We first define the filtration~\eqref{eq:main_filt} on the hairy graph-complex. We say that a hairy graph is of defect zero if it is obtained
by attaching only uni-trivalent trees and only to the edges of its core graph. Otherwise we say that a hairy graph has defect $k>0$ if
it is obtained from a graph of defect zero with the same core by contracting $k$ internal edges. For example, the graph below has defect~4.
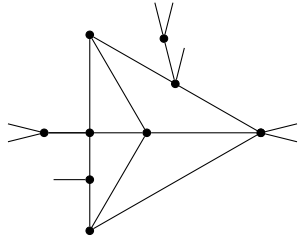
\[
 \begin{tikzpicture}[level distance=6mm, sibling distance=3mm, coo/.style={inner sep=0,outer sep=0}, baseline=-.65ex]
  \node[int] (v) at (0,0) {};
  \node[int] (w1) at (0:1.5) {};
  \node[int] (w2) at (120:1.5) {};
  \node[int] (w3) at (-120:1.5) {};
  \draw (v) edge (w1) edge (w2) edge (w3) 
        (w2) to
         node[midway] (x1) {}
         node[near end](x2) {} 
         (w3) (w2) to
         node[midway] (x3) {} 
         node[near end] (x4) {} 
         node[near start] (x5) {} 
         (w1)
        (w1) to
        node[midway] (x6){} 
        (w3)
        ;
  \node[int] at (x1) {} child[grow=left] { node[int] (y1) {} child { node{} } child {node {}} };
  \node[int] at (x2) {}  child[grow=left] { node{} };
  \node[int] at (w1) {}  
    [grow=right]
      child { node{} } 
      child { node{} } ;
   \node[int] at (x3) {}
    [grow=up]
      child { node{} } 
      child { node[int]{} child {node{}}
              child {node{}} } 
       ;
       \draw (v) edge (y1);
 \end{tikzpicture}
 \]
We define the term $F_i^{h,r}$ of the filtration  to consist of the elements $x\in\HGC_{m,n}^{h,r}$ such that both $x$ and $dx$ 
are linear combinations of graphs of defect $\geq i$. This filtration is induced by the Postnikov filtration in the coefficient systems. 
Recall that for a cochain complex $(C,d)$, its $k$-th  Postnikov term $Po_k(C)$ is the subcomplex of $C$:
\[
(Po_k(C))_i=
\begin{cases}
C_i,& i\leq k;\\
d(C_i),& i=k+1;\\
0,& i>k+1.
\end{cases}
\]
This filtration is functorial and therefore induces a filtration $Po_\bullet(F)$ on any coefficient system $F\colon\Gr_r\to dg\Vect$ and thus on any graph-complex
$\G_F$. On the dual complex $\GC^r_F=G_F^*$ one considers the dual \lq\lq{}orthogonal\rq\rq{} filtration.
One has that $Po_k(C)/Po_{k-1}(C)$ is a two term complex whose cohomology is concentrated in degree $k$ and is exactly $H^k(C)$.
Moreover, one has a functorial in $C$ quasi-isomorphism
\[
H^k(C)  \hookrightarrow Po_k(C)/Po_{k-1}(C).
\]
This implies that the spectral sequence associated with the induced filtation in the graph-complex $\GC^r_F$ has as its first term 
$H(GC_{H(F)})$. Applying this general construction to $\HGC_{m,n}^{h,r}$ described as \eqref{equ:HGCdecoee}-\eqref{equ:HGCdecoeo} and knowing the fact that the homology of $\HPC^h_{A_{n-m-2}}$ is concentrated
only in two degrees, see Subsubsection~\ref{sss332}, we get the statements~\eqref{equ:thm2_1} and \eqref{equ:thm2_2} of the Theorem.

To see that $d_1$ is trivial in case $h=1$, we notice that up to a shift of degree the complex $\HGC_{m,n}^{1,r}$ depends only on the parity of $n$. Thus we get the same splitting as in the case of  $n-m$ even, which implies the collapse of the spectral sequence at $E_1$. 

In case $h=2$, the spectral sequence obviously abuts at $E_1$ as the latter one has only one non-trivial column.

Next consider the claim of Theorem \ref{thm:main_odd} that the Lie bracket with the tripod graph $T$ \eqref{equ:tripodgraph} sends $F_i^{h,r}$ to $F_{i+1}^{h+2,r}$.
%
  The operation of taking the Lie bracket $[T,-]$ is combinatorially the addition of a tripod to each vertex in turn, minus the attachment of a tripod at hairs i.e.,
\[
 [T, \Gamma] = 3\sum_{v\in V\Gamma} \Gamma \cup (\text{tripod at $v$})- 
 \sum_{h\in H\Gamma} \Gamma \cup (\text{tripod at $h$}).
\]
Pictorially, these two operations read
\begin{align}\label{eq:pict_bracket_T}
 \begin{tikzpicture}[baseline=-.65ex, scale = .7]
  \node[int] (v) at (0,.3) {};
  \node at (0,1) {$\cdots$};
  \draw (v)  edge +(0.33,.5) edge +(-0.33,.5) edge +(-0.5,.5) edge +(0.5,.5);
 \end{tikzpicture}
 &\mapsto 
 \begin{tikzpicture}[baseline=-.65ex, scale = .7]
  \node[int] (v) at (0,.3) {};
  \node[int] (w) at (0,-.3) {};
  \node at (0,1) {$\cdots$};
  \draw (v) edge +(0,-.7) edge +(0.33,.5) edge +(-0.33,.5) edge +(-0.5,.5) edge +(0.5,.5) edge (w)
        (w) edge +(-0.5,-.7) edge +(0.5,-.7);
 \end{tikzpicture}
 &
  \begin{tikzpicture}[baseline=-.65ex, scale = .7]
  \node[int] (v) at (0,.3) {};
  \node at (0,1) {$\cdots$};
  \draw (v) edge +(0,-.7) edge +(0.33,.5) edge +(-0.33,.5) edge +(-0.5,.5) edge +(0.5,.5);
 \end{tikzpicture}
 &\mapsto 
 \begin{tikzpicture}[baseline=-.65ex, scale = .7]
  \node[int] (v) at (0,.3) {};
  \node[int] (w) at (0,-.3) {};
  \node at (0,1) {$\cdots$};
  \draw (v) edge +(0,-.7) edge +(0.33,.5) edge +(-0.33,.5) edge +(-0.5,.5) edge +(0.5,.5) edge (w)
        (w) edge +(-0.5,-.7) edge +(0,-.7) edge +(0.5,-.7);
 \end{tikzpicture}
\end{align}
cf. also the graphical description of the Lie bracket \eqref{equ:grbracket}.  This operation increases the defect by~1 (and the number of hairs by~2). Moreover, since $dT=0$, one had $d[T,x]=[T,dx]$. Thus if $dx$ is a linear combinations of graphs of defect $i$, then $d[T,x]=[T,dx]$ is the sum of graphs of defect $i+1$,
which finishes the proof that $[T,-]$ sends $F_i^{h,r}$ to $F_{i+1}^{h+2,r}$. 

Similarly to $[L,-]$ (see Proposition~\ref{P:L}), the map $[T,-]$ can also be described in terms of maps of coefficient systems. It is easier to do so it for the dual coefficient systems. Let $C_{n-m-2}$, respectively $C\rq_{n-m-2}$ be the coalgebra dual to $A_{n-m-2}$, respectively $A\rq_{n-m-2}$. We consider the dual coefficient systems
\begin{gather*}
\HPC^h_{C_{n-m-2}}\colon \Gr_r^{op}\to dg\Vect,\\
\HPC^h_{C\rq_{n-m-2}}\colon \Gr_r^{op}\to dg\Vect
\end{gather*}
defined as objectwise dual of $\HPC^h_{A_{n-m-2}}$, respectively $\HPC^h_{A\rq_{n-m-2}}$.
The operation $[T,-]$ can be extended to a map of coefficient systems
\beq{eq:bracket_T}
[T,-]\colon \HPC^h_{C_{n-m-2}} \Rightarrow \HPC^{h+2}_{C_{n-m-2}}
\eeq
by the same pictorial formulas~\eqref{eq:pict_bracket_T}. We will consider the restriction of~\eqref{eq:bracket_T} on the quasi-isomorphic subsystem
$\HPC^h_{C\rq{}_{n-m-2}} \subset \HPC^h_{C_{n-m-2}} $. We have to show that the induced map of the defect zero homology (coefficient system) of $\HPC^h_{C_{n-m-2}} $ to the defect
one homology (coefficient system) of $\HPC^{h+2}_{C_{n-m-2}}$ is an isomorphism. Since all the restriction maps in the coefficient systems are quasi-isomorphisms, it is
enough to check this statement only for one graph that we choose to be the simplest one, i.e. $W_r$. 

The complex $\HPC^{W_r,h}(C\rq_{n-m-2})$ is dual to~\eqref{eq:derham}. It has length two and can also be described as the de Rham map
 \begin{multline}\label{eq:derham_dual}
  0\to \bigoplus_{l+2l'=h-1} S^l(H^1(W_r)[n-m-2]) \otimes S^{l'}(H^1(W_r)[2n-2m-4])  
  \xrightarrow{\rm d_{dR}} \\
  \bigoplus_{l+2l'=h} H^0(W_r)[n-m-1] \otimes S^l(H^1(W_r)[n-m-2]) \otimes S^{l'}(H^1(W_r)[2n-2m-4]) \to 0.
 \end{multline}
Here we denote the basis of $H^1(W_r)[n-m-2]$ by variables $y_j$  and a basis of $H^1(W_r)[2n-2m-4]$ by variables $\theta_j$. We interpret the spaces in the complex above as polynomials in $y_j$ and $\theta_j$, then the dual ``de Rham'' differential $\rm d_{dR}$ is given by the formula
 \beq{equ:dRDef2}
  \rm d_{dR}f(y_1,\dots,y_n,\theta_1,\dots,\theta_n)= \sum_{j=1}^n \theta_j \frac\partial {\partial y_j} f(y_1,\dots,y_n,\theta_1,\dots,\theta_n).
 \eeq
 The second term $X_2$ of this complex is spanned by graphs $W_r$ that have only edges decorated by hairs (graphs of defect zero), the first term $X_1$ being spanned by 
 similar graphs with the only vertex of $W_r$ decorated by a hair (graphs of defect one). As before a monomial $\theta_j^k$ should be understood as the $j$-th circle
 having $2k$ hairs; and a monomial $y_j\theta_j^k$ should be understood as the $j$-th circle having $2k+1$ hairs. 
 
Define a map $\Psi\colon X_2\to \HPC^{W_r,h+2}(C_{n-m-2})$
 by sending each graph to the sum obtained by attaching two hairs at one of the internal vertices, 
i.e. either in the only vertex $v$ of $W_r$ or in the base of one of the hairs. It is easy to see that $([T,-]|_{X_2}-d\circ\Psi)$  is described as attachment of a hair in $v$,
then taking the differential in $\HPC^{W_r,h+1}(C\rq{}_{n-m-2})$ and then again attachment of a hair 
in $v$. Thus this map abuts in  $\HPC^{W_r,h+2}(C\rq_{n-m-2})
\subset \HPC^{W_r,h+2}(C_{n-m-2})$. This map is essentially $\rm d_{dR}$ that provides an isomorphism in homology: between $\ker \rm d_{dR}$ and $\coker\, \rm d_{dR}$.
\hfill \qed

\begin{rem}[Hodge filtration on $\HP(A_d)$]
Note that for $d=n-m$ odd the coefficient system $\HP(A_d)$ originates from a representation of $\Out(F_r)$ that does not factor through $\GL(r,\Z)$ in general. 
However, we recall from \cite{TW2} or Subsection \ref{sec:HPofAd} that there is a filtration (called Hodge filtration therein) on the Hochschild-Pirashvili homology $\HP(A_d)$, such that the associated graded does factor through the $\GL(r,\Z)$-action on $H_1$. It follows that all the hairy graph cohomology may be considered as a subquotient of the decorated graph cohomology induced by representations of $\GL(n,\Z)$ (on $H_1$).
It might thus be of interest to study these decorated version in more detail.
\end{rem}

\subsection{Consequences and discussion}
Theorems \ref{thm:main_even} and Theorem \ref{thm:main_odd} have interesting consequences for the structure of the hairy graph cohomology.
First, a copy of the non-hairy graph cohomology embeds into $H(\HGC_{m,n})$ in both cases $n-m$ even and $n-m$ odd.
Secondly, the remaining classes ``come in pairs''. In the case $n-m$ even the pairing is realized directly by the bracket with the line graph $L$. This is illustrated on a computer generated table of $H(\HGC_{m,n})$ in Figure \ref{fig:graphtable_even}.

In case $n-m$ odd, we only know for sure that there is a spectral sequence at some of whose page the classes cancel, cf. \cite{TW}.
However, looking at the cancellation pattern in the computer generated table of $H(\HGC_{m,n})$ in Figure \ref{fig:graphtable_odd}, we see that in small degrees the cancellation always seems to happen on the $E_2$ page, and is given by the bracket with the tripod graph $T$.


\begin{figure}
\begin{align*}
&  \begin{tikzpicture}
\matrix (mag) [matrix of math nodes,ampersand replacement=\&]
{
{\phantom{2} } \& 1 \& 2 \& 3 \& 4 \& 5 \& 6 \& 7 \& 8 \&  \&  \&  \&  \&  \&   \\
 9 \& 1_{16} \& 1_{16} \&  \&  \&  \&  \&  \&  \&   \&  \& \& \& \& \\
 8 \&       \& 1_{15} \&  \&  \&  \&  \&  \&  \&   \&  \& \& \& \& \\
 7 \& 1_{12} \& 1_{12} \&  \&  \&  \&  \&  \&  \&   \&  \& \& \& \& \\
 6 \&       \& 1_{11} \&  \&  \&  \&  \&  \&  \&  \&  \&  \&  \& \&  \\
 5 \& 1_8    \&  \&  \&   1_{11} \& \&  \&  \&  \&  \&  \&  \&  \& \&  \\
 4 \&       \&  \&  \& 1_8 \&  \&  \&  \&  \&  \&  \&  \&  \& \&  \\
 3 \& 1_4    \&  \&  \& 1_7 \&  \&  \&  \&  \&  \&  \&  \&  \&  \&\\
 2 \&       \&  \&  \&  \&  \&  \&  \&  \&  \&  \&  \&   \& \& \\
 1 \&  1_0     \&  \& 1_1 \&  \& 1_1 \& 1_4 \& 1_1 \& 1_1,1_4 \&  \&  \&   \& \& \& \\
};
\draw (mag-1-1.south) -- (mag-1-9.south);
\draw (mag-1-1.east) -- (mag-10-1.east);
\draw[-, purple!30, thick] (mag-3-3.center) edge (mag-2-3.center);
\draw[-, purple!30, thick] (mag-5-3.center) edge (mag-4-3.center);
\draw[-, purple!30, thick] (mag-8-5.center) edge (mag-7-5.center);
\end{tikzpicture}
&
\begin{tikzpicture}
\matrix (mag) [matrix of math nodes,ampersand replacement=\&]
{
{\phantom{2} } \& 1 \& 2 \& 3 \& 4 \& 5 \& 6 \& 7 \& 8 \&  \&  \&  \&  \&  \&   \\
 9 \&       \& 2_{14} \&  \&  \&  \&  \&  \&  \&   \&  \& \& \& \& \\
 8 \&  1_{13}  \&  2_{13} \&  \&  \&  \&  \&  \&  \&   \&  \& \& \& \& \\
 7 \&       \& 2_{10} \& 5_{10} \& 1_8 \&  \&  \&  \&  \&   \&  \& \& \& \& \\
 6 \&  1_9 \& 2_9  \& 3_9,1_7 \& 4_7 \&  \&  \&  \&  \&  \&  \&  \&  \& \&  \\
 5 \&       \& 1_6 \& 3_6 \& 6_6  \& 1_4 \&  \&  \&  \&  \&  \&  \&  \& \&  \\
 4 \&  1_5   \& 1_5 \& 2_5 \& 3_5,1_3 \& 4_{5},3_3 \&  \&  \&  \&  \&  \&  \&  \& \&  \\
 3 \&      \&  1_2 \& 1_2 \& 3_2  \& 4_2 \&  \&  \&  \&  \&  \&  \&  \&   \&\\
 2 \&  1_1     \&   1_{1} \& 1_{1} \& 2_{1} \& 2_{1} \&  \&  \&   \&  \& \&  \&   \& \& \\
 1 \&       \& 1_{-2} \& 1_{-2} \& 1_{-2} \& 2_{-2} \& 2_{-2},1_{-5} \& 1_{-5} \&   \&  \&  \&   \& \& \& \\
};
 \draw (mag-1-1.south) -- (mag-1-9.south);
 \draw (mag-1-1.east) -- (mag-10-1.east);
 \draw[-, purple!30, thick] (mag-2-3.center) edge (mag-3-3.center);
 \draw[-, purple!30, thick] (mag-4-3.center) edge (mag-5-3.center);
 \draw[-, purple!30, thick] (mag-6-3.center) edge (mag-7-3.center);
 \draw[-, purple!30, thick] (mag-6-4.center) edge (mag-7-4.center);
  \draw[-, purple!30, thick] (mag-5-4.center) edge (mag-6-4.center);
  \draw[-, purple!30, thick] (mag-4-4.center) edge (mag-5-4.center);
  \draw[-, purple!30, thick] (mag-9-5.center) edge (mag-8-5.center);
  \draw[-, purple!30, thick] (mag-9-6.center) edge (mag-8-6.center);
   \draw[-, purple!30, thick] (mag-9-3.center) edge (mag-8-3.center);
    \draw[-, purple!30, thick] (mag-9-4.center) edge (mag-8-4.center);
\end{tikzpicture}
\end{align*}
 \caption{\label{fig:graphtable_even} Computer generated tables of $\HGC_{m,n}$ in even codimension $n-m$ (left: $n=m=2$, right: $n=m=3$) taken from \cite{KWZ2}, with the cancellations induced by the bracket with the line graph marked. 
 The rows indicate the number of hairs ($\uparrow$), the columns the loop order ($\rightarrow$). A table entry $1_3$ means that there the degree 3 subspace is one-dimensional. }
\end{figure}

\begin{figure}
\begin{align*}
&
\begin{tikzpicture}
\matrix (mag) [matrix of math nodes,ampersand replacement=\&]
{
{\phantom{2} } \& 1 \& 2 \& 3 \& 4 \& 5 \& 6 \& 7 \& 8 \&  \&  \&  \&  \&  \&   \\
 9 \&  1_8 \&  \&  \&  \&  \&  \&  \&  \&   \&  \& \& \& \& \\
 8 \&       \& 1_7 \& 1_7 \&  \&  \&  \&  \&  \&   \&  \& \& \& \& \\
 7 \&       \&  \& 1_8  \&  1_6 \&  \&  \&  \&  \&   \&  \& \& \& \& \\
 6 \&       \& 1_6 \& 1_6 \& 1_6,1_7 \&  \&  \&  \&  \&  \&  \&  \&  \& \&  \\
 5 \&  1_4 \&  \& 1_5 \&  1_5,1_7  \& 2_5 \&  \&  \&  \&  \&  \&  \&  \& \&  \\
 4 \&       \& 1_3 \&  \&  1_5,1_6 \& 1_6 \&  \&  \&  \&  \&  \&  \&  \& \&  \\
 3 \&       \&  \& 1_4 \&  \& 1_4 \&  \&  \&  \&  \&  \&  \&  \&  \&\\
 2 \&       \& 1_2 \&  \&  \& 1_5  \&  \&  \&  \&  \&  \&  \&   \& \& \\
 1 \&  1_0     \&  \& 1_1 \&  \& 1_1 \& 1_4 \& 1_1 \& 1_1,1_4 \&  \&  \&   \& \& \& \\
};
\draw (mag-1-1.south) -- (mag-1-9.south);
\draw (mag-1-1.east) -- (mag-10-1.east);
\draw[-, purple!30, thick] (mag-3-3.center) edge (mag-5-3.center);
\draw[-, purple!30, thick] (mag-7-3.center) edge (mag-9-3.center);
\draw[-, purple!30, thick] (mag-6-4.center) edge (mag-8-4.center);
\draw[-, purple!30, thick] (mag-5-5.center) edge (mag-7-5.center);
\draw[-, purple!30, thick] (mag-3-4.center) edge (mag-5-4.center);
\draw[-, purple!30, thick] (mag-7-6.center) edge (mag-9-6.center);
\draw[-, purple!30, thick] (mag-6-6.center) edge (mag-8-6.center);
\end{tikzpicture}
&
\begin{tikzpicture}
\matrix (mag) [matrix of math nodes,ampersand replacement=\&]
{
{\phantom{2} } \& 1 \& 2 \& 3 \& 4 \& 5 \& 6 \& 7 \& 8 \&  \&  \&  \&  \&  \&   \\
 9 \&       \& 1_6 \&  \&  \&  \&  \&  \&  \&   \&  \& \& \& \& \\
 8 \&       \&  \&  \&  \&  \&  \&  \&  \&   \&  \& \& \& \& \\
 7 \& 1_{5} \& 1_5 \& 2_5 \&  \&  \&  \&  \&  \&   \&  \& \& \& \& \\
 6 \&       \&  \& 1_2 \&  \&  \&  \&  \&  \&  \&  \&  \&  \& \&  \\
 5 \&       \& 1_2 \& 2_2 \& 3_2  \& 5_2 \&  \&  \&  \&  \&  \&  \&  \& \&  \\
 4 \&       \&  \& 1_1 \& 1_1 \& 1_{-1} \&  \&  \&  \&  \&  \&  \&  \& \&  \\
 3 \& 1_1   \& 1_1 \& 2_1 \& 2_1 \& 3_1 \&  \&  \&  \&  \&  \&  \&  \&  \&\\
 2 \&       \&  \&  \&  \& 1_{-2} \& 1_{-2} \&  \&  \&  \&  \&  \&   \& \& \\
 1 \&       \& 1_{-2} \& 1_{-2} \& 1_{-2} \& 2_{-2} \& 2_{-2},1_{-5} \& 1_{-5} \&   \&  \&  \&   \& \& \& \\
};
 \draw (mag-1-1.south) -- (mag-1-9.south);
 \draw (mag-1-1.east) -- (mag-10-1.east);
 \draw[-, purple!30, thick] (mag-2-3.center) edge (mag-4-3.center);
 \draw[-, purple!30, thick] (mag-6-3.center) edge (mag-8-3.center);
 \draw[-, purple!30, thick] (mag-5-4.center) edge (mag-7-4.center);
  \draw[-, purple!30, thick] (mag-6-4.center) edge (mag-8-4.center);
  \draw[-, purple!30, thick] (mag-6-5.center) edge (mag-8-5.center);
  \draw[-, purple!30, thick] (mag-6-6.center) edge (mag-8-6.center);
  \draw[-, purple!30, thick] (mag-7-6.center) edge (mag-9-6.center);
\end{tikzpicture}
 \end{align*}
 \caption{\label{fig:graphtable_odd} Computer generated tables of $\HGC_{m,n}$ in odd codimension (left: $n=2$, $m=1$, right: $n=3$, $m=2$) taken from \cite{KWZ2}, with the cancellations induced by the bracket with the tripod graph marked. }
\end{figure}

\subsection{Remark: String links and a ``colored'' variant}
The construction of the spectral sequences above can be extended to more general ``hairy'' graph complexes considered in the literature.
For example, it is shown in \cite{ST} that under suitable hypothesis the rational homotopy of the space of long embeddings (modulo immersions) of $N$ ``strings'' of dimensions $m_1,\dots,m_N$ in $\R^n$ can be expressed through the graph cohomology of a hairy graph complex $\HGC_{m_1,\dots,m_N;n}$, generalizing the complex $\HGC_{m,n}$ arising in the case $N=1$. Similar graph-complexes were also considered in~\cite{CKV1,CKV2}.\footnote{In these works the construction is
more general  on internal vertices, allowing any cyclic operad as input (commutative operad in our case), but slightly more restrictive on the hair vertices, allowing only even number of colors of the same degrees -- which are basis elements of a symplectic vector space.} 
The complex $\HGC_{m_1,\dots,m_N;n}$ differs in so far that hairs are $N$-colored, with the $j$-colored hairs carrying cohomological degree $m_j$. The complex $\HGC_{m_1,\dots,m_N;n}$ splits according to the loop order $r$ and the number of hairs in each color $\underline{k}=(k_1,\dots,k_N)$. By similar considerations as above we then find a spectral sequence (similarly associated to the filtration by defect) relating the decorated to the (colored) hairy graph cohomology
\[
 H(\GC^r_{\HP^{\underline k}_{A_{\underline{d}}}\otimes \Det^{\otimes n}} )[rn-n] \Rightarrow \HGC_{m_1,\dots,m_N;n}^{\underline{k},r},
\]
where $d=(n-m_1-2,\dots, n-m_N-2)$ and $A_{\underline{d}}$ is as in Subsection \ref{sec:mainexamples}.
As above, the commutative algebra $A_{\underline{d}}$ is formal and may be replaced by its homology $A_{\underline{d}}'$. The first term of the obtained spectral 
sequence is
$H(\GC^r_{\HP^{\underline k}_{A_{\underline{d}}}\otimes \Det^{\otimes n}} )[rn{-}n]$ and is also concentrated on two columns.
Unfortunately, in this case we cannot generally compute $\HP(A_{\underline{d}})$, in contrast to the $N=1$ situation.

\section{Hairy graph-homology in the loop order $r=2$}\label{s:r2}
For the loop order $r=2$, the homology of the hairy graph-complexes $\HGC_{m,n}$ were computed in~\cite{CCTW}. Surprisingly the computations were much harder in the case of even codimension, to which our paper is an explanation. In fact our theorems~\ref{thm:main_even} and~\ref{thm:main_odd} can can substantially simplify 
those computations especially in the difficult case of odd codimension reducing it to the even case. First and a practical remark is that in the graph-complex $\GC_F^r$
for any homotopy coefficient system $F$ we can ignore the graphs with cut vertices. The proof is completely analogous to the case of constant coefficients~\cite[Theorem~1.1]{CCV} (see also~\cite[Theorem~3.1]{CCTW} where this is proved for the hairy graph-complexes). In the case of  loop order~2 there is
only one graph without cut vertices, which we denote by $\theta$:
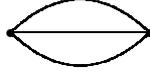
\[
\begin{picture}(60,40)
\put(4,20){\circle*{3}}
\put(56,20){\circle*{3}}
\qbezier(4,20)(30,45)(56,20)
\qbezier(4,20)(30,-5)(56,20)
\put(4,20){\line(1,0){52}}
\end{picture}.
\]
In other words, we get
\[
\GC_F^2\simeq (F^*\otimes\Or)^{G_\theta},
\]
where $G_\theta =S_3\times S_2$ is the group of symmetries of $\theta$. The sign factor $\Or$ is responsible for the permutation of edges, thus it is the sign representation of the factor $S_3$. One can also easily see that in this case $\Or=\Det$. In particular if $V$ is a finite dimensional $\Out(F_2)$ representation
concentrated in degree zero, one has
\beq{eq:H_V_r2}
H^i(\GC_V^2)=
\begin{cases}
0,& i\neq 3;\\
(V^*\otimes \Det)^{G_\theta},& i=3,
\end{cases}
\eeq
The cohomology is concentrated in degree 3, since $\theta$ has 3 edges and corresponds to a 3-dimensional orbicell of $\OS_2$. From equation~\eqref{eq:H_V_r2} one immediately gets
\beq{eq:S_2k-1}
H^3(GC_{S^{2k-1}H_1}^2)=H^3(GC_{\Det\otimes S^{2k-1}H_1}^2)=0.
\eeq
Indeed, the symmetry of $\theta$ that preserves edges and flips the vertices acts as $-1$ on $H_1$ and thus as $(-1)^{2k-1}=-1$ on $S^{2k-1}H_1$, while $\Det$ produces a positive sign.  This in particular means that the splitting $\HGC_{m,n}^{2,h}=\HGC_{m,n}^{2,h,I}\oplus \HGC_{m,n}^{2,h,II}$ of Theorem~\ref{thm:main_even} 
is trivial in homology in the sense that one of the two terms of the splitting is always acyclic. Computations made in~\cite{CCTW}, specifically its~\cite[Theorems~6.1 and~6.2]{CCTW},
imply
\beq{eq:S_2k}
H^3(\GC_{S^{2k}H_1}^2)=\K^{\lfloor{\frac k3}\rfloor}; \qquad H^3(\GC_{\Det\otimes S^{2k}H_1}^2)=\K^{\lfloor{\frac k3}\rfloor+1}.
\eeq

A similar situation takes place in odd codimension as well. The spectral
sequence of Theorem~\ref{thm:main_odd} always abuts at the first term for $r=2$ as this term $E_1$ always has only one non-trivial column. Indeed, 
$\Out(F_2)=\GL(2,\Z)$ and thus representations $B_2^h$ obviously factor through $\GL(2,\Z)$ (contrary to the case $r\geq 3$). Moreover, easy computations
show that 
\beq{eq:Bh_2}
B_2^{2k-1}\simeq S^k H_1\quad  \text{ and } \quad B_2^{2k}\simeq \Det\otimes S^{k-1} H_1.
\eeq
This together with~\eqref{eq:S_2k-1} implies that one of the two columns in $E_1$ is always zero. Theorem~\ref{thm:main_odd} together 
with~(\ref{eq:H_V_r2}-\ref{eq:Bh_2}) recover the computations \cite[Theorems~6.3 and 6.4]{CCTW} of the hairy graph-homology for $r=2$ in the odd case.

\section{Application to the deformation theory of the little discs operads}\label{s:operads}
As we mentioned in the introduction, the hairy graph-complexes $\HGC_{m,n}$ appear naturally in the relative deformation theory of the little discs operads. In this section we briefly recall how exactly they appear and also we explain how our main results Theorems~\ref{thm:main_even} and~\ref{thm:main_odd} are related and actually give a simpler proof to some earlier results of the authors obtained in~\cite{grt,TW} about the relative deformations of little discs operads in codimensions~0 and~1. 

In this section we use freely the language of operads. A good introduction into the subject can be found in the textbook \cite{lodayval}, whose conventions we mostly follow. 

We use the notation $\op P\{k\}$ for the $k$-fold operadic desuspension.
The operads governing commutative, associative and Lie algebras are denoted by $\Com$, $\Ass$ and $\Lie$ respectively.
The $n$-Poisson operad $\Poiss_n$ governs (non-unital) commutative algebras with an additional Lie bracket of degree $1-n$, which is a derivation with respect to the commutative product.
It contains as a sub-operads $\Com$ and the desuspended Lie operad $\Lie_n:= \Lie\{n-1\}$.

For a coaugmented cooperad $\op C$ we denote by $\Omega(C)$ its cobar construction, cf. \cite[section 6.5]{lodayval}.
If $\op P$ is a quadratic Koszul operad, we denote its Koszul dual cooperad by $\op P^\vee$.
In this case we often use the notation $\mathsf{ho}\op{P}$ for the cobar construction of the Koszul dual, e.g.,
\begin{align*}
 \hoAss &= \Omega(\Ass^\vee) & \hoPois_n &= \Omega(\Pois_n^\vee) & \hoLie_n=\Omega(\Lie_n^\vee) & &\text{etc...}
\end{align*}
The quadratic operads considered above are well-known to be Koszul. One has $\Ass^\vee =\Ass^*\{1\}$ (the cooperad dual to $\Ass$ operadically desuspended once); $\Com^\vee=\Lie^*\{1\}$; $\Lie_n^\vee=\Com^*\{n\}$; $\Pois_n^\vee=\Pois_n^*\{n\}$.

For any morphism of dg operads ${\mathcal P}\to {\mathcal Q}$, one can define the deformation complex $\Def({\mathcal P}\to {\mathcal Q})$ which is the complex
of derivations of the composite map $\hat {\mathcal P}\to {\mathcal P}\to {\mathcal Q}$ (where $\hat {\mathcal P}\to {\mathcal P}$ is a cofibrant replacement of ${\mathcal P}$) shifted in degree by one, so that it is endowed with a natural $L_\infty$ structure~\cite{KS,lodayval}. In case $\hat {\mathcal P}$ is a cobar construction 
of a dg cooperad, the induced $L_\infty$ structure is  a dg Lie algebra structure.

By $E_n$ we denote the operad of singular chains of the little discs operads. Its homology operad is the associative operad $\Ass$ if $n=1$, and the operad $\Poiss_n$
if $n>1$. 
It has been shown in~\cite{FW,TW} that the natural inclusion $E_m\to E_n$ is rationally a formal map of operads if and only if $n-m\neq 1$.\footnote{See also~\cite{LV}, where this result has been established earlier for a weaker range of dimensions.}  Thus in case $n-m\neq 1$, the deformation complex $\Def(E_m\to E_n)$ is equivalent to the deformation complex of the induced map of operads in homology. In particular, in case $n>m+1>2$, 
\[
\Def(E_m\to E_n)\simeq \Def(\Pois_m\xrightarrow{*}\Pois_n),
\]
where $ \Pois_m\xrightarrow{*}\Pois_n$ denotes the composite map $\Pois_m\to\Com\to\Pois_n$ of the obvious projection followed by an inclusion.

Kontsevich's operad $\Graphs_n$ (see \cite{K2}) is an operad whose space of $r$-ary operations $\Graphs_n(r)$ consists of linear combinations of isomorphism classes of graphs with $r$ numbered ``external'' vertices and an arbitrary number of unnumbered ``internal'' vertices.
These graphs are required to satisfy the additional conditions:
\begin{itemize}
 \item All internal vertices have at least valence $3$.
 \item Each connected component contains at least one external vertex.
\end{itemize}
The following picture shows an examples of such an admissible graph.
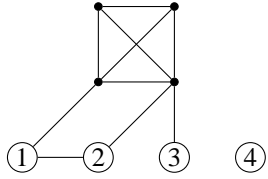
\[
\begin{tikzpicture}
 \node[ext] (v1) at (0,0) {1};
 \node[ext] (v2) at (1,0) {2};
 \node[ext] (v3) at (2,0) {3};
 \node[ext] (v4) at (3,0) {4};
 \node[int] (w1) at (1,1) {};
 \node[int] (w2) at (1,2) {};
 \node[int] (w3) at (2,1) {};
 \node[int] (w4) at (2,2) {};
 \draw (v1) edge (v2) edge (w1) 
       (w2) edge (w1) edge (w3) edge (w4)
       (w3) edge (v3) edge (v2) edge (w1) edge (w4)
       (w1) edge (w4);
\end{tikzpicture}
\]
The cohomological degree of a graph $\Gamma$ is the number 
\[
 n (\#(\text{internal vertices})-1) - (n-1) \#(\text{edges}).
\]
For more details, signs, and the definition of the differential and the operad structure we refer the reader to \cite{KS,LV,TW}.

The key result is that the operad $\Graphs_n$ forms a model for the homology operad of the little $n$-disks operad for $n\geq 2$.
\begin{thm}[Kontsevich \cite{K2}, Lambrechts-Voli\'c \cite{LV}]
 There is a quasi-isomorphism of operads $\Pois_n\to \Graphs_n$ for all $n\in \mathbb{Z}$.
\end{thm}

We may now consider the deformation complex 
\beq{equ:Defcomplex}
 \Def(\hoPois_m \xrightarrow{*} \Graphs_n),
\eeq
where the map to be deformed is the composition 
\[
 \hoPois_m \to \Pois_m \to \Com \to \Graphs_n.
\]
Concretely, as a graded vector space the above deformation complex is isomorphic to 
\[
 \Def(\hoPois_m \xrightarrow{*} \Graphs_n) \cong \prod_{r\geq 1} \Hom_{S_r}(\Pois_m^*\{m\}(r), \Graphs_n(r)).
\]

One defines the sub-complex 
\beq{equ:fHGCDef}
\fHGC_{m,n} \subset  \Def(\hoPois_m \xrightarrow{*} \Graphs_n)
\eeq
to be spanned by maps that satisfy the following two conditions:
\begin{itemize}
 \item The image is a series of graphs all of whose external vertices are univalent.
 \item The map factors through $\Pois_m^*\{m\}\to \Com^*\{m\}$.
\end{itemize}
Elements of $\fHGC_{m,n}$ are naturally identified with series of (not necessarily connected) ``hairy'' graphs as depicted in \eqref{equ:HGCsample}.
The cohomological degree of such a graph is computed as the number 
\[
  n \#(\text{internal vertices}) - (n-1) \#(\text{edges}) + m (\#(\text{hairs})-1).
\]
We cite the following result from the literature.
\begin{thm}[\cite{Turchin2}, \cite{TW}]
 The inclusion \eqref{equ:fHGCDef} is a quasi-isomorphism.
\end{thm}

Any hairy graph is naturally a union of its connected pieces.
Hence we may identify $\fHGC_{m,n}$ with the completed symmetric product space  of the connected subcomplex 
\[
 \HGC_{m,n} \subset \fHGC_{m,n} = S^+\left( \HGC_{m,n}[-m]\right) [m].
\]
(The sign $+$ in $S^+$ means that  the degree zero term is omitted.) Finally, let us note that the deformation complex \eqref{equ:Defcomplex} above is naturally a dg Lie algebra, and the subspaces 
\[
 \HGC_{m,n} \subset \fHGC_{m,n} \subset  \Def(\hoPois_m \xrightarrow{*} \Graphs_n)
\]
are closed under the Lie bracket, so that the spaces $\HGC_{m,n}$ and $\fHGC_{m,n}$ carry natural dg Lie algebra structures. The Lie bracket in question is  
 described by~\eqref{equ:grbracket}, see~\cite{TW}.

\subsection{Little discs deformations in codimensions 0 and 1}
As we explained above, because of the relative formality of the little discs operads in codimension $n-m>1$, for any field $\K$ of characteristic zero, the deformation complex $\Def(E_m\to E_n)$
is quasi-isomorphic to $\fHGC_{m,n}= S^+\left( \HGC_{m,n}[-m]\right) [m]$.  Recall Theorem~\ref{thm:loop01} which describes the loop orders zero and one  of the
hairy graph-homology $H(\HGC_{m,n})$.  Theorems~\ref{thm:main_even} and~\ref{thm:main_odd} deal with the loop order $r\geq 2$ of the hairy graph-homology.
These two theorems are respectively related to the two results~\cite[Theorem~1.3]{grt} and~\cite[Theorem~4]{TW}  obtained earlier by the authors:

\begin{thm}[Willwacher~\cite{grt}]\label{thm:codim0}
For $n\geq 2$,
\beq{eq:codim0}
H(\Def(E_n\xrightarrow{id} E_n)) \simeq H(\Def(\Pois_n\xrightarrow{id} \Pois_n)) =
S^+\left(\K[-n-1]\oplus V_n [-n-1] \oplus \prod_{r\geq 2} H(\GC^r_{\Det^{\otimes n}})[nr-n-1]\right)[n],
\eeq
where $\K[-n-1]$ is a one-dimensional space  spanned by the only class in loop order zero; $V_n$ stands for the space of classes of loop order one
(so called wheels):
\beq{eq:Vn}
V_n=\bigoplus_{{j\geq 1}\atop {j\equiv 2n+1\mod 4}} \K[n-j];
\eeq
the rest is the product whose $r$-th term describes the homology of loop order $r$.
\end{thm}

Willwacher proves the above theorem by showing that the deformation complex $\Def(\Pois_n\xrightarrow{id} \Pois_n)$ 
is equivalent to the full hairy graph-complex $\fHGC_{n,n}$ with deformation being the initial differential (sum of expansions of internal vertices) 
plus the bracket with the line graph $L$.  This perturbation on the level of generators does not change the homology in loop order zero and one, but 
 in loop order $r\geq 2$ it kills all the repeated terms described by Theorem~\ref{thm:main_even}, leaving only the part of the homology arising from 
 $\HGC_{n,n}^{r,1,II}$. 
 
 \begin{thm}[Turchin and Willwacher~\cite{TW}]\label{thm:codim1}
For $n\geq 2$,
\beq{eq:codim1}
H(\Def(E_{n-1}\to E_n)) \simeq
S^+\left(\K[-n]\oplus V_n [-n] \oplus \prod_{r\geq 2} H(\GC^r_{\Det^{\otimes n}})[nr-n]\right)[n-1],
\eeq
where $\K[-n]$ is a one-dimensional space  spanned by the only class (the tripod $T$) in loop order zero; $V_n$ stands for the space of classes of loop order 
one~\eqref{eq:Vn};
the rest is the product whose $r$-th term describes the homology of loop order $r$.
\end{thm}

 Again the proof is obtained by showing that $\Def(E_{n-1}\to E_n)$ is quasi-isomorphic to $\fHGC_{n-1,n}$  with the differential perturbed by a certain Maurer-Cartan element which contains as a summand the tripod graph $T$. This explains why on the level of generators in loop order $r\geq 2$ all the repeated terms described by Theorem~\ref{thm:main_odd} are killed. (Theorem~\ref{thm:main_odd} could simplify or rather enlighten the proof of Theorem~\ref{thm:codim1} given in~\cite{TW}, where we used a different combinatorial approach to tackle this problem.) We also warn the reader that strictly speaking the loop order in  $\Def(E_{n-1}\to E_n)$  is defined only as a filtration as the Maurer-Cartan element in question might have graphs of a positive loop order. The fact that one gets  a non-trivial cancellation in the deformation homology was used in~\cite{TW} to prove  that the map of operads  $E_{n-1}\to E_n$ is not formal.

\subsection{Relative versus target deformations}\label{ss:rel_target}
One could consider the natural map
\beq{eq:rel_target}
\Def(E_n\xrightarrow{id} E_n)\to \Def(E_m\to E_n)
\eeq
and the induced map in homology. In other words the question is to describe how the deformations of the target can be seen as a part of relative deformations of the little discs operads. 

The hardest case
is when the map $E_m\to E_n$ is not formal, i.e. when $m=n-1$. In this case the authors gave a partial answer~\cite[Theorem~2]{TW} which states that under the 
map~\eqref{eq:rel_target}, the space of generators of~\eqref{eq:codim0} under the cup product is sent isomorphically to the space of generators of~\eqref{eq:codim1}.\footnote{In other words, the $E_{n-1}$ deformations are rigid within the $E_n$ structure being all induced by the $E_n$ deformations. We call this fact
{\it algebraic Cerf lemma} in~\cite{TW}.} 
 We also believe, but were not able to prove, that the elements obtained by taking cup product in~\eqref{eq:codim0} (symmetric powers of order $\geq 2$) are sent to zero.  
 
 Theorems~\ref{thm:main_even}-\ref{thm:main_odd} help to solve this problem in the easy case $n-m\geq 2$.

\begin{thm}\label{thm:rel_target}
For $n-m\geq 2$, the map
\beq{eq:Hrel_target}
H(\Def(E_n\xrightarrow{id} E_n))\to H(\Def(E_m\to E_n))
\eeq
sends
\begin{itemize}
\item the only loop order zero generator of~\eqref{eq:codim0} to zero;
\item  the loop order one generators to zero except the 1-wheel, which is sent to the graph-cycle 
 \begin{tikzpicture}[scale=.2]
 \draw (0,0) circle (1);
 \draw (-180:1) node[int]{} -- +(-1.2,0);
 \end{tikzpicture} (this happens when $n$ is even, otherwise all 1-loop classes with no exception are sent to zero);
 \item all the loop order $\geq 2$ generators of~\eqref{eq:codim0} isomorphically to the corresponding piece of the homology of 
 $\HGC_{m,n}^{r,1,II}$;\footnote{Notice that the complex $\HGC_{m,n}^{r,1}$ does not depend on $m$, and thus the splitting of Theorem~\ref{thm:main_even}
 $\HGC_{m,n}^{r,1} = \HGC_{m,n}^{r,1,I}\oplus \HGC_{m,n}^{r,1,II}$ takes place for the odd codimension as well. In terms of graphs, the map~\eqref{eq:Hrel_target}
 sends a bald graph to the sum of graphs obtained by attaching a hair in one of its vertices.}
 \item all the elements obtained by taking a cup-product (i.e. elements from any symmetric power of order $\geq 2$) to zero.
 \end{itemize}
 \end{thm}
 
\begin{proof}
Because of the relative formality result it is enough to study the map
\beq{eq:rel_target_formal}
\Def(\Pois_n\xrightarrow{id}\Pois_n)\to \Def(\Pois_m\xrightarrow{*}\Pois_n).
\eeq
(In fact for $m=1$ we need to consider the operad $\Ass$ instead of $\Pois_1$, but the argument given below will still work as $\Def(\Ass\xrightarrow{*}\Pois_n)$ is also equivalent to the full hairy graph-complex $\fHGC_{1,n}$, see~\cite{Turchin1}.) Since the map
$\Pois_m\xrightarrow{*}\Pois_n$ factors through the commutative operad $\Com$, the map~\eqref{eq:rel_target_formal} can itself be
viewed as a composition
\beq{eq:rel_target_fcomp}
\Def(\Pois_n\xrightarrow{id}\Pois_n)\to \Def(\Com\to \Pois_n)\to \Def(\Pois_m\xrightarrow{*}\Pois_n).
\eeq
The complex $\Def(\Com\to\Pois_n)$ is modeled by the complex of derivations of the map $\hoCom\xrightarrow{*}\Graphs_n$ with the latter one by a similar argument quasi-isomorphic to the subcomplex of $\HGC_{m,n}\subset\fHGC_{m,n}$ spanned by the graphs with only one hair. This subcomplex
$\prod_{r\geq 1} \HGC_{m,n}^{r,1}$ does not depend on $m$ and will be denoted by $\HGC_n^1$. The second map in~\eqref{eq:rel_target_fcomp} is thus modeled
by the inclusion $\HGC_n^1\hookrightarrow\fHGC_{m,n}$. On the other hand, the complex $\Def(\Pois_n\xrightarrow{id}\Pois_n)$ is modeled by the complex 
of derivations of the quasi-isomorphism $\hoPois_n\xrightarrow{\simeq}\Graphs_n$, which according to~\cite{grt} is quasi-isomorphic to $\HGC_{n,n}^L$ -- the 
hairy graph-complex $\HGC_{n,n}$ with the differential deformed by the Maurer-Cartan element $L$. One has that the square
\begin{equation}\label{eq_sq}
\xymatrix{
  \fHGC_{n,n}^L\ar[r]^-\simeq\ar[d] &\Def(\hoPois_n\xrightarrow{\simeq}\Graphs_n)\ar[d] \\
   \HGC_n^1\ar[r]^-\simeq & \Def(\hoCom\xrightarrow{*}\Graphs_n)
   }
  \end{equation}  
  commutes. Thus the first map in~\eqref{eq:rel_target_fcomp} is modeled by the projection
  \[
  \fHGC_{n,n}^L\to \HGC_{n}^1,
  \]
  sending all graphs with $\geq 2$ hairs to zero and the graphs with exactly one hair to themselves.
  From this explicit description of the map~\eqref{eq:rel_target} as the composition 
  \[
  \fHGC_{n,n}^L\to \HGC_{n}^1 \to \fHGC_{m,n}
  \]
  the result easily follows.
  \end{proof}
  
  \begin{rem}
  One can also ask how the deformations of the source are seen in the relative deformations of the little discs operads or in other words one can look at the induced map in the homology of the natural map
  \[
  \Def(E_m\xrightarrow{id} E_m)\to \Def(E_m\to E_n).
  \]
  In case of codimension $n-m\geq 2$ this map is trivial as it factors through $\Def(E_m\xrightarrow{*} \Com)$ which has trivial homology.
  \end{rem}
  
%

\section{Application: Serre fibrations whose fiber is a wedge of circles}\label{s:serre_fibr}

Lemma~\ref{l:splitting} about the splitting of the coefficient system $C_\bullet\cong \K\oplus \tilde C_\bullet$ is related to a curious topological phenomenon
that the analogue of the Euler class for fibrations of wedges of $\geq 2$ circles is always rationally trivial and as a consequence we prove that the following 
result holds.

\begin{thm}\label{thm:serre_fibr}
The rational (co)homology Serre spectral sequence associated to any Serre fibration whose fiber is homotopy equivalent to a wedge of $r\geq 2$ circles abuts at the second page.
\end{thm}

We start with the following lemma. The proof of the theorem is given in Subsection~\ref{ss:proof_serre}.

\begin{lemma}\label{l:monoid}
For any $r\geq 2$, all the components of the monoid $G(\vee_r S^1)$ of self homotopy equivalences of $\vee_r S^1$ are weakly contractible and $\pi_0 G(\vee_r S^1)=
\Out(F_r)$.
\end{lemma}

\begin{proof}
Consider the Serre fibration
\beq{eq:taking_bp}
G(\vee_r S^1)\to \vee_r S^1,
\eeq
obtained by taking the image of the base point. Its fiber is the monoid $G_*(\vee_r S^1)$ of pointed self-homotopy equivalences. The latter space has all components contractible as this is true for $\Omega(\vee_r S^1)$. Thus, $G_*(\vee_r S^1)\simeq \Aut(F_r)$. The homotopy fiber of the map $G_*(\vee_r S^1)\to G(\vee_r S^1)$
is $\Omega(\vee_r S^1)\simeq F_r$ by the general properties of fibration sequences. On the other hand, for $r\geq 2$, the map $F_r\to \Aut(F_r)$ (each element $x$ is sent to the conjugation by $x$) is injective, thus its fiber (which is equivalent to $\Omega G(\vee_r S^1)$ from our fibration sequence) is contractible.  Thus the connected component of the identity (and therefore every component of $G(\vee_r S^1)$ since its $\pi_0$ is a group) is weakly contractible. From the long exact sequence for  fibration~\eqref{eq:taking_bp},
we get $\pi_0 G(\vee_r S^1)=\Out(F_r)$ and $\pi_i G(\vee_r S^1)=0$ for $i\geq 1$. 
\end{proof}

\subsection{Analogue of the Euler class}\label{ss:euler}
Recall that for a Serre fibration $F\to E\to B$ whose fiber $F$ is a circle, the Euler class $e\in H^2(B,H_1(F))$ is defined as the obstruction to the existence of a section.    
     In fact it is an obstruction of extension of a section from the 1-skeleton of $B$ to its 2-skeleton. (It does not depend on the choice of the section over the 1-skeleton.) In case $F\simeq \vee_r S^1$ the obstruction of a section can be similarly defined as a certain class $e_r\in H^2(B,H_1(F))$. In particular if a section exists, then this class is zero. More generally for any ring $\K$ of coefficients and assuming that B is a regular $CW$ complex, if for any cell of $B$ one can find a formal $\K$ linear combination $\sum_i\alpha_i s_i$, with $\sum_i \alpha_i=1$,  of sections $s_i$  which glue together compatibly on the boundary, then the  image $e_r^\K$ of this class in $H^2(B,H_1(F,\K))$ must also be zero. We call such an object a $\K$-linear section. In case $\K=\Q$ we call it also a rational section.

 \begin{thm}\label{thm:euler}
 For any Serre fibration $F\to E\to B$ with $F$ homotopy equivalent to a wedge of $r\geq  2$ circles, the class $e_r^\Q\in H^2(B,H_1(F,\Q))$ is trivial.
 \end{thm}
 
 \begin{proof}
 Serre fibrations $F\to E\to B$ with fiber $F\simeq \vee_r S^1$ are classified by maps $f\colon B\to B G(\vee_r S^1)$, where, for $r\geq 2$, the latter space is equivalent to $B\Out(F_r)$ by Lemma~\ref{l:monoid}. The obstruction class $e_r^\Q$  for $F\to E\to B$ is the pullback of the analogous class that we denote by ${\mathsf e}_r$ of the canonical $\vee_r S^1$ bundle over $B\Out(F_r)$.  Thus it is enough to show that ${\mathsf e}_r^\Q$ is  trivial.
 
 Rationally, instead of $B\Out F_r$ one can use the orbispace $\OS_r$ -- the moduli space of graphs. Whilst the role of the canonical $\vee_r S^1$ fibration is played by the Serre orbifibration $\pi S\colon W\OS_r\to \OS_r$, where $W\OS_r$ is the space of pairs $(\Gamma,x)$ with $\Gamma$ a metric graph in $\OS_r$ and $x$  a point in 
 $\Gamma/G_\Gamma$ (here $G_\Gamma$ is as before the group of symmetries of $\Gamma$).  It is a Serre orbifibration in the sense that it is obtained as the quotient by the $\Out(F_r)$ of the actual Serre fibration $\pi\colon  W\OX_r\to \OX_r$, where similarly $W\OX_r$ is the space of pairs -- a graph $\Gamma$ in $\OX_r$ (i.e. a graph together with a class of a homotopy equivalence $\vee_r S^1\to \Gamma$) and a point $x$ in $\Gamma$. The space $W\OX_r$ has an obvious $\Out(F_r)$ action, such that $\pi$ is $\Out(F_r)$-equivariant and one gets $W\OX_r/\Out(F_r) = W\OS_r$. 
 
 The orbifibration $\pi S$ admits a rational orbi-section assigning to an element $\Gamma\in \OS_r$ a linear combination of the vertices of $\Gamma$ (as in the proof of Lemma~\ref{l:splitting}):
 \[
 \sum_{v \in V\Gamma} \frac {|v|-2}{2r-2} v.
 \]
 (Again it is an orbi-section in the sense that it is obtained from an $\Out(F_r)$-equivariant section of $\pi$, given by the same formula.)
 
 \end{proof}
 
 \begin{rem}\label{r:orbi_not_nec}
 In fact we can avoid the \lq\lq{}orbi-langauge\rq\rq{} in the proof by considering instead of $\pi S$  the actual Serre fibration 
 \[
 (W\OX_r\times E\Out(F_r))/\Out(F_r) \to (\OX_r\times E\Out(F_r))/\Out(F_r),
 \]
 where $E\Out(F_r)$ is a weakly contractible space with a free $\Out(F_r)$ action such that its quotient by $\Out(F_r)$ is $B\Out(F_r)$. The fibers of the map above are again metric graphs, thus a rational section can be defined similarly.
 \end{rem}

 \subsection{Proof of Theorem~\ref{thm:serre_fibr}}\label{ss:proof_serre}
 Consider the Serre rational cohomology spectral sequence for a fibration $F\to E\to B$ with $F\simeq \vee_r S^1$, $r\geq 2$. Its second term $E_2^{p,q}=
 H^p(B,H^q(F,\Q))$ is concentrated on two horizontal lines $q=0$ and~1. The differential $d_2\colon H^p(B,H^1(F,\Q))\to H^{p+2}(B,H^0(F,\Q))$ is the composition
 \[
 H^p(B,H^1(F,\Q))\xrightarrow{e_r^\Q \cup (-)} H^{p+2}(B,H^1(F,\Q)\otimes H_1(F,\Q))\to H^{p+2}(B,\Q) = H^{p+2}(B,H^0(F,\Q)),
 \]
 where the first map is the cup-product with the analogue of the Euler class considered in the previous subsection, and the second map is induced by the obvious homomorphism of the coefficient systems. By Theorem~\ref{thm:euler}, $e_r^\Q=0$, and thus $d_2=0$. All other $d_i$, $i\geq 3$, are trivial by dimensional reason.

%
%
%
%
%


\bibliographystyle{plain}

\end{document}